\documentclass[reqno]{amsart}
\usepackage[margin=3cm]{geometry}
\usepackage{amsmath}
\usepackage{amssymb}
\usepackage{amsbsy}
\usepackage[latin1]{inputenc}
\usepackage{dsfont}
\usepackage{color}

\usepackage{enumerate}

\def\f        {{\boldsymbol f}}
\def\g        {{\boldsymbol g}}
\def\n        {{\boldsymbol n}}
\def\u        {{\boldsymbol u}}
\def\v        {{\boldsymbol v}}
\def\w        {{\boldsymbol w}}
\def\x        {{\boldsymbol x}}
\def\y        {{\boldsymbol y}}

\def\W        {{\boldsymbol W}}
\def\L        {{\boldsymbol L}}
\def\H        {{\boldsymbol H}}
\def\V        {{\boldsymbol V}}

\def\ds{{\rm d}s}
\def\dt{{\rm d}t}

\def\R {{\mathds R}}
\def\N {{\mathds N}}

\newtheorem{theorem}{Theorem}[section]
\newtheorem{lemma}[theorem]{Lemma}

\newtheorem{corollary}[theorem]{Corollary}
\newtheorem{definition}[theorem]{Definition}

\newtheorem{remark}[theorem]{Remark}

\begin{document}

\title[Convergence to suitable weak solutions for a subgrid FEM model]{Convergence to suitable weak solutions for a finite element approximation of the Navier-Stokes equations with numerical subgrid scale modeling}

\author[S. Badia]{ Santiago Badia$^\dag$}

\thanks{$\dag$ Universitat Polit\`ecnica de Catalunya, Jordi Girona1-3, Edifici C1, E-08034 Barcelona $\&$ Centre Internacional de M\`etodes Num\`erics en Enginyeria, Parc Mediterrani de la Tecnologia, Esteve Terrades 5, E-08860 Castelldefels, Spain E-mail: {\tt sbadia@cimne.upc.edu}.  SB was partially supported by by the European Research Council under the FP7 Program Ideas through the Starting Grant No. 258443 - COMFUS: Computational Methods for Fusion Technology and the FP7 NUMEXAS project under grant agreement 611636. SB gratefully acknowledges the support received from the Catalan Government through the ICREA Acad\`emia Research Program. }

\author[J. V. Gutiérrez-Santacreu]{Juan Vicente Guti\'errez-Santacreu$^\ddag$}
\thanks{$\ddag$ Dpto. de Matemática Aplicada I, E. T. S. I. Informática, Universidad de Sevilla. Avda. Reina Mercedes, s/n. E-41012 Sevilla, Spain.  E-mail: {\tt juanvi@us.es}. JVGS was partially supported by the Spanish grant No. MTM2015-69875-P from Ministerio de Economía y Competitividad with the participation of FEDER}

\date{\today}

\begin{abstract} In this work we prove that weak solutions constructed by a variational multiscale method are suitable in the sense of Scheffer. In order to prove this result, we consider a subgrid model that enforces orthogonality between subgrid and finite element components. Further, the subgrid component must be tracked in time. Since this type of schemes introduce pressure stabilization, we have proved the result for equal-order velocity and pressure finite element spaces that do not satisfy a discrete inf-sup condition.
\end{abstract}

\maketitle



\tableofcontents

\section{Introduction}

Incompressible Newtonian fluids are governed by the Navier-Stokes
equations. The existence of solutions is known from the works by Leray
\cite{Leray} and Hopf \cite{Hopf}. However, uniqueness is still an open
question. The loss of regularity is related to turbulence
\cite{guermond_use_2008}, and Leray denoted weak solutions as
\emph{turbulent solution}. Scheffer defined the concept of suitable
weak solutions in \cite{Scheffer_1977} and proved a bound for the
Haussdorff dimension of the singular set of a weak suitable
solution. This result was later improved by Cafarelli, Kohn, and
Nirenberg \cite{Caffarelli-Kohn-Nirenberg}, proving that this dimension
is smaller than 1. This is the sharpest regularity result so far.

Suitable weak solutions of the Navier-Stokes equations can be
constructed by regularization (see, e.g., \cite{Lions_1969}). More
recently, Guermond proved that inf-sup stable finite element (FE)
approximations having a discrete commutator property also converge to
suitable weak solutions, first for periodic boundary conditions in the
three-dimensional torus \cite{guermond_finite-element-based_2006}, and
next on general domains and no-slip boundary conditions
\cite{Guermond_2007}. The Fourier method does not satisfy the required
assumptions, and it is still an open question whether it provides
suitable solutions.

The Navier-Stokes equations have a dissipative structure, due to the
viscous term. The system has a singular limit in the assymptotic
regime as the Reynolds (Re) number, which is the ratio of inertia
forces to viscous forces, goes to infinity. The singular limit and the
fact that the system is indefinite complicate its numerical
approximation. The first property requires to introduce some kind of
convection stabilization, whereas the second prevents the use of the same
FE space for both the velocity and pressure unknowns, the discrete
system is unstable.

At the continuous level, the nonlinear convective term transfers
energy from the largest to the smallest scales, till reaching the
Kolmogorov scale, where energy is dissipated. In direct numerical
simulations (DNS) the mesh needs to be fine enough to capture the smallest
scales in the flow. However, this approach is unacceptable for
industrial turbulent flows, due to the limits in computational
resources. In real applications, under-resolved simulations are
needed. The smallest scales that can be captured in these simulations
are far from the Kolmogorov scale and dissipation is negligible. Thus,
one has to add so-called large eddy simulation (LES) turbulent models
that add artificial diffusion mechanisms. The concept of suitability
and the fact that energy is dissipated at the mesh scale in a
physically consistent way have been related in
\cite{guermond_use_2008}. Otherwise, an energy pile-up occurs at the
smallest grid scales, leading to instabilities.

Convection stabilization and turbulence models are strongly
related. In this sense, many authors have considered so-called
implicit LES (ILES) methods that do not modify the original
Navier-Stokes equations but introduce additional numerical artifacts
when carrying out the discretization \cite{Boris1992,Grinstein2007}. In
the frame of FE techniques, one approach is to consider variational
multiscale (VMS) methods \cite{hughes_variational_1998, Hughes2000}. The
idea is to use a two-scale decomposition of the original problem and
provide a numerically motivated closure for the fine scale (see, e.g.,
\cite{guermond_stabilization_1999}). A similar stabilization procedure
can be used for the convective term and the pressure term, leading to
methods that do not require to satisfy a discrete inf-sup
condition. An alternative to traditional residual-based methods is to
consider subscales that are in some sense orthogonal to the FE
space. This idea has been proposed by Codina \cite{codina_time_2007},
where $L^2(\Omega)$ orthogonality was used. This method involves
global projections, which has motivated the use of local projections
(see, e.g., \cite{becker_finite_2001, Badia2012}). The treatment of the
time dimension in the subgrid model has also been object of active
research. In particular, the use of dynamic subscales methods that
track the subgrid scale in time have been proposed in
\cite{codina_time_2007}.

Even though DNS is impractical in real applications, it is better
understood than stabilized or ILES schemes. The groundbreaking works
by Guermond have proved that the FE Galerkin method leads to weak
suitable solutions in
\cite{guermond_finite-element-based_2006,Guermond_2007}. However, the
extension to ILES methods is not straightforward, due to the
introduction of additional terms to the numerical formulation. The
analysis of these methods has usually been restricted to \emph{a
  priori} error estimates for smooth enough solution (see, e.g.,
\cite{codinab5}).  Residual-based VMS schemes are not amenable for weak
convergence analysis, due to the proliferation of terms, e.g.,
including new velocity-pressure coupling terms. However, enforcing the
modelled subgrid scales to be orthogonal to the FE space and
considering the dynamic formulation in \cite{codina_time_2007}, the
authors have proved in \cite{Badia-Gutierrez} that the resulting scheme
converges to weak (turbulent) solutions of the Navier-Stokes
equations. For the same scheme, long-term stability estimates and
existence of a global attractor have been proved in
\cite{Badia-Codina-Gutierrez}. Further, a very detailed numerical
experimentation of these methods for isotropic and wall-bounded
turbulent flows can be found in \cite{colomes_assessment_2015}, proving
that these subgrid models act as accurate turbulence
models. Theoretical analyses supporting these results can also be
found in \cite{guasch_statistical_2013}.

In this work, we want to analyze whether VMS-type FE ILES schemes
converge to suitable weak solutions in the sense of Scheffer. We prove
that subgrid closures that are orthogonal and dynamic converge in fact
to suitable solutions for equal order FE pairs for the velocity and
pressure unknowns.

The outline of the work is the following. First, we state the problem
and introduce the notation in Section \ref{sec:statement}. The FE
approximation based on the VMS-type ILES scheme is introduced in Section
\ref{sec:fe}. Section \ref{sec:tech} includes some technical results
in fractional Sobolev spaces. Energy estimates are proved in Section
\ref{sec:energy}. Finally, the convergence towards weak and suitable
solutions is proved in \ref{weaksol}.

\section{Statement of the problem} \label{sec:statement}
Throughout this paper we follow faithfully the notation used in \cite{Guermond-Pasciak_2008} and \cite{Guermond_2007} so that the reader can trace with ease the main differences between these two works and the one presented herein. 
\subsection{Notation} Let $\Omega$ be an open subset of $\R^3$. For $p\in[1,\infty]$, we denote by $L^p(\Omega)$ the usual Lebesgue space, i.e.,  
$$
L^p(\Omega) = \{v : \Omega \to \R\, :\, v \mbox{ Lebesgue-measurable}, \int_\Omega |v(\x)|^p d\x<\infty \},
$$
with the usual modification when $p=\infty$. This space is a Banach space endowed with the norm
$\|v\|_{L^p(\Omega)}=(\int_{\Omega}|v(\x)|^p\,{\rm d}\x)^{1/p}$ if $p\in[1, \infty)$ or $\|v\|_{L^\infty(\Omega)}={\rm ess}\sup_{\x\in \Omega}|v(\x)|$ if $p=\infty$. In particular,  $L^2(\Omega)$ is a Hilbert space.  We shall use  
$\left(u,v\right)=\int_{\Omega}u(\x)v(\x){\rm d}\x$ for its inner product and $\|\cdot\|$ for its norm.  
For $m\in\N$, we denoted by $H^m(\Omega)$ the classical Sobolev-Hilbert spaces, i.e., 
$$
H^{m}(\Omega) = \{v \in L^2(\Omega)\,:\, \partial^k v \in L^2(\Omega)\ \forall ~ |k|\le m\}
$$ associated to the norm 
$$
\|v\|_{H^{m}(\Omega)} =
\left[\sum_{0\le |k| \le m} \|\partial^k v\|^2_{L^{2}(\Omega)}\right]^{\frac{1}{2}}\,,
$$ where $k = (k_1,...,k_d)\in{\mathbb{N}^d}$ is a multi-index and $|k| = \sum_{i=1}^d k_i$. Let $\mathcal{D}(\Omega)$ be the space of infinitely times differentiable functions with compact support in $\Omega$, i.e. the space of test functions on $\Omega$. Thus $H^{m}_0(\Omega)$ is defined as the completion of ${\mathcal D}(\Omega)$ with respect to the $H^{m}(\Omega)$-norm. Fractional-order Hilbert-Sobolev spaces are defined by the real method or $K$-method of interpolation due to Peetre and Lions \cite{Adams_Fournier}. Thus, we consider two spaces: $H^s(\Omega)=[L^2(\Omega), H^1(\Omega)]_{s}$, for $s\in (0,1)$, and  $\tilde H^s_0(\Omega)=[L^2(\Omega), H^1_0(\Omega)]_{s}$ for $s\in[0,1]$. Moreover, for $s\in(0,1)$, $H^s_0(\Omega)$ is the closure of $\mathcal{D}(\Omega)$ with respect to the $H^s(\Omega)$-norm. Note that the spaces $H^s(\Omega)$ and $H^s_0(\Omega)$ coincide for $s\in [0,\frac{1}{2}]$, with uniform norms \cite[Th 11.1]{Guermond_2009}, and the spaces $H^s(\Omega)$ and $\tilde H^s_0(\Omega)$ coincide with equivalent norms \cite{Lions_Magenes_1968} for $s\in[0,\frac{1}{2})$.  We also consider $H^s(\Omega)=[H^1(\Omega), H^2(\Omega)]_{s}$  for $s\in (1,2]$ and $\tilde H^s_0(\Omega)=H^s(\Omega)\cap H^1_0(\Omega)$ for $s\in (1,2]$.

The dual space of $\mathcal{D}(\Omega)$, the space of distributions, is denoted by $\mathcal{D}' (\Omega)$. Moreover, for $s<0$, $\tilde H^s(\Omega)$ is the dual of $\tilde H^{-s}_0(\Omega)$ and the space $H^{-s}_0(\Omega)$ is the complexion of $\mathcal{D}(\Omega)$ under the norm
$$
\|v\|_{H^{-s}(\Omega)}=\sup_{w\in \mathcal{D}(\Omega)\backslash\{\boldsymbol{0}\}}\frac{(v,w)}{\|w\|_{H^s(\Omega)}},
$$
We use  $\langle \cdot, \cdot\rangle$ to denote the duality pairing. For $s\in[0,\frac{1}{2})\cup (\frac{1}{2},\frac{3}{2})$, $H^{-s}(\Omega)$ coincides with $\tilde H^{-s}_0(\Omega)$.

We will use boldfaced letters for spaces of vector functions, e.g. $\L^2(\Omega)$ in place of $L^2(\Omega)^d$.

We will make use of the following space of vector fields:
$$\boldsymbol{\vartheta}=\{\v\in \boldsymbol{\mathcal{D}}(\Omega): \nabla\cdot\v=0 \mbox{ in } \Omega \}. $$
Related to the space $\boldsymbol{\vartheta}$, we consider the closures
in the $\L^2(\Omega)$ and $\H^1(\Omega)$-norm,  which are characterized by
$$
\begin{array}{lll}
\H&=& \{ \u \in \L^2(\Omega) : \nabla\cdot\u =0 \mbox{ in } \Omega, \u\cdot\boldsymbol{n} = 0 \hbox{
on }
\partial\Omega \},
\\
{\V}&=& \{\u \in \H^1(\Omega) : \nabla\cdot\u =0 \mbox{ in } \Omega, \u = {\bf 0}
\hbox{ on }
\partial\Omega \},
\end{array}
$$
where $\n$ is the outward normal to $\Omega$ on $\partial \Omega$. This characterization is true for locally Lipschitz-continuous domains (see \cite[Theorems 1.4 and 1.6]{Temam_2001} for a detailed proof). Furthermore, $L^2_{\int=0}(\Omega)$ (resp. $H^1_{\int =0}(\Omega)$)  is the space of zero-average $L^2(\Omega)$-functions (resp. zero-average $H^1(\Omega)$-functions ). Thus, by the real method of interpolation, $H^s_{\int=0}(\Omega)=[L^2_{\int=0}(\Omega), H^1_{\int=0}(\Omega)]$ for $s\in (0,1)$ (see \cite{Guermond_2009}). 

Let $X$ be a Banach space. Thus, $L^p(a,b;X)$ denotes the space of Bochner-measurable, $X$-valued functions on the interval $(0, T)$ such that $\int_0^T\|f(s)\|^p_{X} {\rm d} s<\infty$ if $1\le p<\infty$ or  ${\rm ess}\sup_{s\in(0,T)}\|f(s)\|_X<\infty$ if $p=\infty$. 
 
Moreover, $W^{1,1}(0,T; X)$ is the space of functions $f\in L^1(0,T; X)$ and $\frac{{\rm d}}{ {\rm d}s}f\in L^1(0,T; X) $ such that $\int_0^T(\|f(s)\|_X+\|\frac{{\rm d}}{{\rm d}s}f(s)\|_X)\,\ds<\infty$ and $W^{1,1}_0(0,T; X)$ is the closure of ${\mathcal D}(0, T; X)$   with respect to  the $W^{1,1}(0,T; X)$-norm, with $\mathcal{D}(0, T;X)$ being the space of infinitely times differentiable  functions defined on $(0,T)$ having values into $X$ with compact support in $(0,T)$.  Additionally, the dual  space of $W^{1,1}_0(0,T; X)$  is denoted by $W^{-1, \infty}(0,T; X')$ provided that $X$ is separable and reflexive.

The Fourier transform of a function $f \in L^1(\R; X)$ is  denoted by
$$
\mathcal{F}f (\xi):=\int_{-\infty}^{+\infty} e^{-  2\pi i t\cdot\xi} f(t) dt.
$$
Let $H$ be a Hilbert space and let $\mathcal{S}'(\R; H)$ be the space of tempered distributions taking value in $H$. Thus, for $\gamma\in\R$, one defines
$$
H^{\gamma} (\R; H)=\{ v\in \mathcal{S}'(\R; H); \int_\R (1+|\xi|)^{2\gamma} \|\mathcal{F} v\|_H^2 d\xi\},
$$
where $H$ is a Hilbert space. Additionally, the space $H^\gamma(0,T;H)$ is made up of tempered distributions in $\mathcal{S}'(0,T;H)$ with the norm
$$
\| v \|_{H^\gamma (0,T; H) }= \inf_{\mathfrak{v}\in \mathcal{S}'(\R;H) }\| \mathfrak{v}\|_{H^\gamma(\R;H)},
$$
where $\mathfrak{v}$ is the extension of $v$ by zero off $(0,T)$ belonging to $\mathcal{S}'(\R;H)$. 

Note that throughout this paper we use the symbol $C$ (with or without subscripts) to represent generic positive constants which can take different values at different places. 

\subsection{The Navier-Stokes equations}
The Navier-Stokes equations for the motion of a viscous, incompressible, Newtonian fluid can be written as
\begin{equation}\label{NS}
\left\{
\begin{array}{rclcc}
\displaystyle
\partial_t\u-\nu \Delta \u + (\u\cdot \nabla) \u  +
\nabla p &=&  \f & \mbox{ in }& \Omega\times (0, T),
\\
\nabla\cdot \u&=&0 &\mbox{ in } & \Omega\times (0, T),
\end{array}
\right.
\end{equation}
with $ \Omega$ being a bounded, three-dimensional domain and with $0<T<+\infty$. Here $\u: \Omega\times (0,T)\to \R^3$ represents the incompressible fluid velocity and $p: \Omega\times (0,T)\to \R $ represents the fluid pressure. Moreover, $\f$ is the external body force which acts on the system, and $\nu > 0$ is the kinematic fluid viscosity.

These equations are supplemented by the no-slip boundary condition
\begin{equation}\label{boundary-condition}
\u={\boldsymbol 0}\quad \mbox { on } \quad  \partial \Omega\times (0,T),
\end{equation}
and the initial condition
\begin{equation}\label{initial-condition}
\u(0)=\u_0\quad\mbox{ in }\quad \Omega.
\end{equation}

The first authors dealing with the concept of \emph{weak} solutions for the Navier-Stokes equations were Leray \cite{Leray} for the Cauchy problem in the whole space and later Hopf \cite{Hopf} for the initial-boundary value problem in bounded domains. Particularly, weak solutions were called \emph{turbulent} by  Leray due to the possible connection between the lack of regularity of weak solutions and turbulence.    
\begin{definition}\label{def:weak-solution} A function $\u$ is said to be a weak solution of
problem (\ref{NS})-(\ref{boundary-condition}) if:
\begin{equation}\label{regularity}
\u\in L^\infty(0,T; \H)\cap L^2(0,T; \V)
\end{equation}
and
$$
\begin{array}{l}
\displaystyle -\int_0^T(\u(t),\partial_t\v(t))\,\dt+\int_0^T
\langle(\u(t)\cdot\nabla)\u(t),\v(t)\rangle\,\dt +\int_0^T\nu(\nabla\u(t),\nabla\v(t))\,\dt
\\
\displaystyle =(\u_0, \v(0)) +\int_0^T\langle\f(t),\v(t)\rangle\,\dt
\end{array}
$$
for all $\v\in {W}^{1,1}(0,T;\V)$ with $\v(T)=\bf{0}$. Moreover, the energy
inequality
\begin{equation}\label{Energy}
\frac{1}{2}\|\u(t)\|^2+\nu\int_{0}^{t}\|\nabla\u(s)\|^2\,\ds\le
\frac{1}{2}\|\u_0\|^2+\int_{0}^t\langle\f(s),\u(s)\rangle\,\ds
\end{equation}
holds  a. e.~in $[0,T]$.
\end{definition}

An equivalent definition for weak solutions involving the pressure term is defined as follows.  
\begin{definition}\label{def:weak-solution2}
A pair $(\u, p)$ is said to be a weak solution of
problem (\ref{NS})-(\ref{boundary-condition}) if:
\begin{equation*}
\u\in L^\infty(0,T; \H)\cap L^2(0,T; \V)\quad \mbox{ and }\quad p\in W^{-1, \infty}(0,T, L^2(\Omega)/\R)
\end{equation*}
and
$$
\left\{
\begin{array}{l}
\partial_t\u+(\u\cdot\nabla)\u-\nu\Delta \u+\nabla p =\f\quad \mbox{ in } W^{-1, \infty}(0,T;\H^{-1}(\Omega)),
\\
\u(0)=\u_0 \quad\mbox{ in } \quad \H.
\end{array}
\right.
$$
Moreover, the energy inequality
$$
\frac{1}{2}\|\u(t)\|^2+\nu\int_{0}^{t}\|\nabla\u(s)\|^2\,\ds\le
\frac{1}{2}\|\u_0\|^2+\int_{0}^t\langle\f(s),\u(s)\rangle\,\ds
$$
holds  a. e.~in $[0,T]$.
\end{definition}

We refer the reader to \cite[Th. 1.3, Ch. V]{Girault-Raviart_1979} for a proof of the equivalence between Definitions \ref{def:weak-solution} and \ref{def:weak-solution2} with $p\in \mathcal{D}'((0,T)\times\Omega)$, that can easily be extended to $p\in W^{-1,\infty}(0,T; L^2(\Omega)/\R)$, by using de Rham's Lemma in \cite[Lm. 2]{Simon_1990}.

The two previous definitions of weak solutions can be proved for $\Omega$ being a bounded, Lipschitzian domain, and $\f\in L^2(0,T; \H^{-1}(\Omega))$ only. The weak solution that will be proved in this paper requires $\Omega$ to be, for instance, convex, and $\f\in L^2(0, T+1; \H^{-1}(\Omega))\cap L^p(0,T+1;\L^q(\Omega))$, with $p\in[1,2]$ and $q\in[1,\frac{3}{2}]$ satisfying $\frac{2}{p}+\frac{3}{q}=4$.

\begin{definition}\label{def:weak-solution3}
A pair $(\u, p)$ is said to be a weak solution of
problem (\ref{NS})-(\ref{boundary-condition}) if:
$$
\u\in L^\infty(0,T; \H)\cap L^2(0,T; \V)\quad \mbox{ and }\quad p\in H^{- r }(0,T, H^{1-s}_{\int=0}(\Omega))
$$
with $s\in(\frac{1}{2}, \frac{7}{10}]$ and $r>\bar r=\frac{3}{4}-\frac{s}{2}$, and
$$
\left\{
\begin{array}{rcccl}
\partial_t\u+(\u\cdot\nabla)\u-\nu\Delta \u+\nabla p &=&\f & \mbox{ in }& H^{- r}(0,T;\tilde\H^{-s}_0(\Omega)),
\\
\u(0)&=&\u_0& \mbox{ in } & \H.
\end{array}
\right.
$$
Moreover, the energy inequality
$$
\frac{1}{2}\|\u(t)\|^2+\nu\int_{0}^{t}\|\nabla\u(s)\|^2\,\ds\le
\frac{1}{2}\|\u_0\|^2+\int_{0}^t\langle\f(s),\u(s)\rangle\,\ds
$$
holds  a. e.~in $[0,T]$.
\end{definition}

Scheffer \cite{Scheffer_1977} introduced the definition of suitable weak solutions so as to prove a partial regularity theorem. Afterwards,  Caffarelli, Kohn, and Nirenberg \cite{Caffarelli-Kohn-Nirenberg} improved Scheffer's results, and F.-H. Lin \cite{Lin_1998} simplified the proofs of the results in \cite{Caffarelli-Kohn-Nirenberg}. 

\begin{definition}\label{Suitable_solution} A weak solution $(\u, p) $ is said to be suitable if the local energy inequality 
$$
\partial_t(\frac{1}{2}\u^2)+\nabla\cdot((\frac{1}{2}\u^2+p)\u)-\nu \Delta (\frac{1}{2}\u^2)+\nu(\nabla\u)^2-\f\cdot\u\le 0
$$
holds in $\mathcal {D}' ((0,T)\times \Omega; \R^+)$.
\end{definition}

\section{Finite element approximation}\label{sec:fe}
\subsection{Hypotheses}
Throughout this paper we will assume the following hypotheses: 
\begin{enumerate}
\item[(H1)] Let $\Omega $ be a connected, bounded, open subset of  $\R^3$ having a polyhedral boundary  such that there exist $\v\in \V\cap\H^2(\Omega)$ and $p\in H^1_{\int=0}(\Omega)$ satisfying
$$
\begin{array}{rcccl}
-\Delta\v+\nabla p&=&\g&\mbox{ in }&\Omega,
\\
\nabla\cdot\v&=&0&\mbox{ in } & \Omega.
\end{array}
$$ 
\item [(H2)] Consider $\{{\mathcal T}_h\}_{h>0}$ to be a shape-regular and quasi-uniform family of simplicial and conforming meshes of $\Omega$ such that $\overline{\Omega}=\cup_{K\in {\mathcal T}_h } K$ with $h=\max_{K\in{\mathcal T}_h} h_K$ where $h_K=\mathrm{ diam } \,K$. 
\item [(H3)] Let  $\{\W_h\}_{h>0}$ and $\{Q_h\}_{h>0}$ be  two families of finite-element spaces associated with  $\{{\mathcal T}_h\}_{h>0}$ such that $\W_h\subset \H^1_0(\Omega)$ and  $Q_h\subset H^1_{\int=0}(\Omega)$. Moreover, the finite-element spaces are required to satisfy the following conditions. Let $\pi_{\W_h}:\L^2(\Omega)\to \W_h$ and $\pi_{Q_h}:\L^2(\Omega)\to Q_h$  be the orthogonal projections onto $\W_h$ and $Q_h$, respectively,  with respect to the $\L^2(\Omega)$-inner product. Furthermore, we denote $\pi_{\W_h}^\perp(\cdot) := (\cdot) - \pi_{\W_h}(\cdot)$ and $\pi_{Q_h}^\perp(\cdot) := (\cdot) - \pi_{Q_h}(\cdot)$. 
\begin{enumerate}[(a)]
\item  There exists a constant $C_{\rm inv}>0$, independent of $h$, such that, for all $\w_h\in\W_h$, 
\begin{equation}\label{inverseL^inf-L^k}
\|\w_h\|_{\L^\infty(\Omega)}\le C_{\rm inv} h^{-\frac{3}{k}} \|\w_h\|_{\L^k(\Omega)}
\end{equation}
and
\begin{equation}\label{inverse-GradL^k-L^k}
\|\nabla\w_h\|_{\L^k(\Omega)}\le C_{\rm inv} h^{-1} \|\w_h\|_{\L^k(\Omega)}
\end{equation}
for $k\in[2,\infty]$,

\begin{equation}\label{inverseH^1-H^s}
\|\w_h\|_{\H^1(\Omega)} \le C_{\rm inv} h^{-1+s} \|\w_h\|_{\tilde\H^s_0(\Omega)}  
\end{equation}
for each $s\in[0,1]$, and 
\begin{equation}\label{inverseH^s-L^2andL^2-H^-s}
\|\w_h\|_{\tilde \H^s_0(\Omega)}\le C_{\rm inv} h^{-s} \|\w_h\|\quad\mbox{ and }\quad  \|\w_h\|\le C_{	\rm inv} h^{-s} \|\w_h\|_{\tilde \H^s_0(\Omega)}
\end{equation}
for $s\in[0,1]$. 
\item There exists a constant $C_{\rm st}(s)>0$, independent of $h$, such that, for $s\in[0,\frac{3}{2})$,
\begin{equation}\label{stability}
\|\pi_{\W_h} \w\|_{\tilde\H^s_0(\Omega)}\le C_{\rm st} (s) \| \w \|_{\tilde\H^s_0(\Omega)}\quad \mbox{ for all }\quad \w\in\tilde\H^s_0(\Omega),
\end{equation}
\item There exists a constant $C_{\rm int}>0$, independent of $h$, such that, for all $l$ and $s$, satisfying $0\le l \le \min\{1,s\}$ and $l\le s \le 2$, there holds
\begin{equation}\label{approx-velocity}
\|\pi_{\W_h}^\perp\w\|_{\tilde\H^l_0(\Omega)}\le C_{\rm int} h^{s-l} \|\w\|_{\tilde\H^{s}_0(\Omega)}\quad\mbox{ for all }\quad\w\in \tilde\H^s_0(\Omega),
\end{equation}
and
\begin{equation}\label{approx-pressure}
\|\pi^\perp_{Q_h}q\|_{H^l(\Omega)}\le C_{\rm int} h^{s-l} \|q\|_{H^{s}(\Omega)} \quad\mbox{ for all }\quad q\in H^s_{\int=0}(\Omega).
\end{equation}
\item There exists $C_{\rm com}>0$, independent of $h$, such that, for $0\le l \le m\le 1$ and $\varphi\in W^{2,\infty}_0(\Omega)$, 
\begin{equation}\label{comm-prop-vel}
\|\pi^\perp_{\W_h}(\varphi \w_h)\|_{\H^l(\Omega)}\le C h^{1+m-l} \|\w_h\|_{\H^m(\Omega)} \|\varphi\|_{W^{m+1,\infty}_0(\Omega)}\quad\mbox{ for all }\quad \w_h\in\W_h, 
\end{equation}
and
\begin{equation}\label{comm-prop-pre}
\|\pi_{Q_h}^\perp(\varphi q_h)\|_{\H^l(\Omega)}\le C h^{1+m-l} \|q_h\|_{\H^m(\Omega)} \|\varphi\|_{W^{m+1,\infty}_0(\Omega)}\quad\mbox{ for all }\quad q_h\in Q_h.
\end{equation}
\end{enumerate}
\item[(H4)] Let $\u_0\in\V$ and $\f\in L^2(0,T+1; \H^{-1}(\Omega))\cap L^p(0,T+1;\L^q(\Omega))$, with $p\in[1,2]$ and $q\in[1,\frac{3}{2}]$ satisfying $\frac{2}{p}+\frac{3}{q}=4$.  
\end{enumerate}

Hypothesis $(\rm H1)$ is ensured for domains having a $C^{1,1}$ boundary or being a convex polygon (cf. \cite{Kellogg-Osborn_1976} or \cite{Grisvard_1985}) or polyhedron (cf. \cite{Dauge_1989} ), with continuous dependence on $\f$. 

Hypothesis $\rm (H3)$ is extremely flexible and allows equal-order finite-element spaces for velocity and pressure. For instance, let  $\mathcal{P}_k(K)$ be the set of piecewise polynomial functions of degree less than or equal to $k$ on  $K$ being a tetrahedra. Thus the space of continuous, piecewise polynomial functions of degree less than or equal to $k$  on a mesh ${\mathcal T}_h$  is denoted as
$$
X_h = \left\{ v_h \in {C}^0(\overline\Omega) \;:\; v_h|_K \in \mathcal{P}_k(K), \  \forall K \in \mathcal{T}_h \right\},
$$
We choose the following continuous finite-element spaces 
$$
\W_h=\boldsymbol{X}_h\cap\H^1_0(\Omega)\quad\hbox{and}\quad Q_h=X_h\cap L^2_{\int=0}(\Omega),
$$
for approximating velocity and pressure, respectively.

The shape-regular and quasi-uniform properties of $\{{\mathcal T}_h\}_{h>0}$ assumed in $(\rm H2)$ suffice to ensure the properties of $(\rm H3)(a)$. We recommend the books  \cite[Sec. 4.5 ]{Brenner-Scott} and \cite[Sec. 1.7]{Ern-Guermond_2004} for a proof of \eqref{inverseL^inf-L^k} and  \eqref{inverse-GradL^k-L^k}, Appendix \ref{App-A} for a proof of \eqref{inverseH^1-H^s}, and \cite{Girault-Raviart_1986} for a proof of \eqref{inverseH^s-L^2andL^2-H^-s}. Moreover, the error estimates stated in (\rm H3) make use of $(\rm H2)$ as well (see \cite[Lm A.3, Rm 2.1]{Guermond-Pasciak_2008} for a proof).

The local approximation properties for the orthogonal projection operators $\pi_{\W_h}$ and $\pi_{Q_h}$ guarantee hypothesis $(\rm H4)$. The reader is referred to \cite{Bertoluzza_1999}.

\begin{remark} Let $p$ and $q$ be as in $(\rm H4)$. We know from Sobolev's embeddings that $\tilde\H^s_0(\Omega)$ is embedded in $\L^{q'}(\Omega)$, where $\frac{1}{q'}+\frac{1}{q}=1$ and $s=3(\frac{1}{q}-\frac{1}{2})$; hence $\L^q(\Omega)$ is embedded in $\tilde\H^{-s}_0(\Omega)$. Moreover, $H^r(\R; H) $ is embedded in $L^{p'}(\R; H)$, where $\frac{1}{p'}+\frac{1}{p}=1$ and $r >\bar r=\frac{1}{p}-\frac{1}{2}$ with $H$ being a Hilbert space; hence $L^{p}(\R; H) $ is embedded in $H^{-r}(\R; H) $ .  Let $\boldsymbol{\mathfrak{f}}$ be the extension of $\f$ outside $[0,T]$ as zero. Then, by Hausdorff-Young's inequality for the Fourier transform, we have
\begin{equation}\label{rm3.1-lab1}
\|\mathcal{F}\boldsymbol{\mathfrak{f}}\|_{H^{-r} (\R; \tilde\H^{-s}_0(\Omega))}\le C \|\mathcal{F}\boldsymbol{\mathfrak{f}}\|_{L^{p'} (\R; \L^{q}(\Omega))}\le C  \|\boldsymbol{\mathfrak{f}}\|_{L^{p} (\R; \L^{q}(\Omega))}= C \|\f\|_{L^{p} (0,T ; \L^{q}(\Omega))}.
\end{equation}
Therefore, 
\begin{equation}\label{rm3.1-lab2}
\f\in H^{-r}(0,T; \tilde\H^{-s}_0(\Omega)) 
\end{equation}
As a reference for further development, it is well to point out, here, the conditions for $p$, $q$, $s$ and $\bar r$:
\begin{enumerate}
\item [({\rm C})]
 Let  $s=3(\frac{1}{q}-\frac{1}{2})$ and $\bar r=\frac{1}{p}-\frac{1}{2}$  be defined for $p$ and  $q$ as in  ({\rm H4}). 
\end{enumerate}
\end{remark}

\subsection{The discrete problem}
Find $\u_h \in H^1(0,T;  \W_h)$, $p_h\in L^2(0,T; Q_h)$ and $\tilde\u_h\in H^1(0,T;\tilde \W_h)$ such that, for all $(\v_h,\tilde\v_h, q_h)\in \W_h\times\tilde\W_h\times Q_h$, 
\begin{subequations}\label{eq:GalerkinNS}
\begin{equation}\label{eq:Galerkin-u}
\left\{
\begin{array}{l}
\displaystyle (\partial_t\u_h,\v_h)+b(\u_h,\u_h,
\v_h)+\nu(\nabla\u_h,\nabla\v_h)
\\
\displaystyle
-(p_h,\nabla\cdot\v_h)- b(\u_h,\v_h,\tilde\u_h)=(\f_h,\v_h),
\end{array}
\right.
\end{equation}
\begin{equation}\label{eq:Galerkin-p}
(\u_h,\nabla q_h)+(\tilde\u_h,\nabla q_h)=0,
\end{equation}
\begin{equation}\label{eq:Sub-Scale}
\left\{
\begin{array}{l}
\displaystyle
(\partial_t\tilde\u_h,\tilde\v_h)+b(\u_h,\u_h,\tilde\v_h)
\\
\displaystyle
+\tau^{-1}(\tilde\u_h,\tilde\v_h)+(\nabla
p_h,\tilde\v_h)=0,
\end{array}
\right.
\end{equation}\label{eq:Initial_Condition}
\begin{equation}
\u_h(0)=\u_{0h}, 
\end{equation}
\end{subequations}
where $$\tau=\frac{1}{\frac{C_
s\nu}{h^2}+\frac{C_c\|\u_h\|_{L^\infty(\Omega)}}{h}}=\frac{h^2}{C_s\nu+C_c h\|\u_h\|_{\L^\infty(\Omega)}},$$  with  $C_s$ and $C_c$ being algorithmic positive constants, and $\f_h\in\W_h$ is defined by duality as $(\f_h,\w_h)=\langle\f,\w_h\rangle$, for all $\w_h\in\W_h$. Let us define
$$
b(\u_h,\v_h, \w_h)= \langle \mathcal{N}(\u_h,\v_h), \w_h \rangle,
$$
where $\mathcal{N}(\u_h,\v_h) = (\u_h \cdot \nabla) \v_h + \frac{1}{2}(\nabla \cdot \u_h) \v_h.$

Let $\{\boldsymbol{\psi}_i\}_{i=1,...,{\rm n}_{\rm u}}$ be a basis of $\W_h$ and let  $\{ \psi_i\}_{i=1,...,{\rm n}_{p}}$ be a basis of $Q_h$, where ${\rm n}_{\rm u}$ and ${\rm n}_{p}$ denote the space dimension for $\W_h$ and $Q_h$, respectively.  Thus, one defines
$$
\tilde \W_h = {\rm span} \{ \pi_{\W_h}^\perp(\mathcal{N}(\boldsymbol{\phi}_i,
\boldsymbol{\phi}_j)), \pi_{\W_h}^\perp(\nabla \phi_k) \},
$$
and  $\W_\star=\W_h \oplus \tilde\W_h$. Moreover, one defines 
$$
\V_\star=\{\v_\star\in\W_\star : (\v_h, \nabla q_h)+(\tilde\v_h, \nabla q_h)=0 \mbox{ for all } q_h\in Q_h\}.
$$
which is a non-conforming approximation space of $\V$.

The initialization of the discrete problem can be obtained by the
following projection problem: find $\u_{0h} \in \V_h$, $\tilde\u_{0h}
\in \tilde{\V}_h$ and $\xi_h \in Q_h$ such that
\begin{equation}
\left\{
\begin{array}{rcll}\label{proj_u0}
\left( \u_{0h},\v_h\right) - \left(\xi_h,\nabla \cdot \v_h \right)
&= &(\u_0,\v_h),& \mbox{ for all }\quad \v_h \in \V_h,
\\
\left( \tilde\u_{0h},\tilde\v\right) + \left(\nabla \xi_h, \tilde \v
\right) &=& (\u_0,\tilde \v_h),& \mbox{ for all } \quad \tilde{\v}_h \in \tilde{\V}_h,
\\
\left( \nabla\cdot\u_{0h}, q_h \right) - \left( \tilde\u_{0h}, \nabla q_h \right)& =& 0, &\mbox{ for all } \quad q_h \in Q_h.
\end{array}
\right.
\end{equation}

\subsection{Discrete operators} 
This subsection is devoted to introducing the discrete operators that are used throughtout this paper. 

Firstly, we will consider a conforming and non-conforming approximation of the Laplace operator $-\Delta: \tilde\H^2_0(\Omega) \to \L^2(\Omega)$. The non-conforming approximation is based on a stabilizing technique. 

Consider $-\Delta_h:\H^1_0(\Omega)\to\W_h$ to be the discrete Laplacian operator defined as:
$$
-(\Delta_h\w, \bar\w_h)=(\nabla\w_h, \nabla\bar\w_h) \quad\mbox{ for all }\quad \bar\w_h\in\W_h.
$$ 
The restriction of this operator $-\Delta_h$  to $\W_h\subset\H^1_0(\Omega)$ gives a self-adjoint, positive-definite operator. Therefore, we are allowed to define the fractional power of $-\Delta_h$, say $(-\Delta_h)^s$, for all $s\in \R$, by the Hilbert-Schmidt theorem. The domain of definition of $(-\Delta_h)^s$ is $ D((-\Delta_h)^s)\equiv \W_h$ since $\dim \W_h<\infty$. Hence, $\W_h^s$ makes reference to $\W_h$ equipped with the Hilbert norm  
$$
\|\w_h\|_{\W^s_h}=((-\Delta_h)^\frac{s}{2}\w_h, (-\Delta_h)^\frac{s}{2}\w_h)^{\frac{1}{2}}.
$$             
The family $\{\W_h^s\}_{s\in\R}$ is a scale of Hilbert spaces with respect to the real method of interpolation.  
Analogously, consider $-\Delta_\star : \W_\star\to \W_\star$ to be the stabilized discrete Laplacian operator defined as  
$$
-(\Delta_\star\w_\star, \bar\w_\star)=(\nabla \pi_{\W_h} \w_\star, \nabla \pi_{\W_h}\bar\w_\star)+h^{-2} (\pi^{\perp}_h\w_\star, \pi^{\perp}_h\bar\w_\star)\quad\mbox{ for all }\quad \bar\w_\star\in\W_\star.
$$
It is easy to see that $-\Delta_\star\w_\star=-\pi_{\W_h}\Delta_\star\w_\star-\pi_{\W_h}^\perp\Delta_\star\w_\star =-\Delta_h\pi_{\W_h}\w_\star- h^{-2}\pi^\perp_{\W_h}\w_\star$. We have that $-\Delta_\star$ is self-adjoint and positive-definite. Therefore, we are also allowed to define the fractional power of $-\Delta_\star$, say $(-\Delta_\star)^s$, for all $s\in\R$, by the Hilbert-Schmidt theorem. Thus, $\W_\star^s$ is $\W_\star$ equipped with the Hilbert norm
$$
\|\w_\star\|_{\W^s_\star}=((- \Delta_\star)^{\frac{s}{2}}\w_\star, (-\Delta_\star)^\frac{s}{2}\w_\star).
$$

Secondly, we will consider a non-conforming approximation of the Stokes operator $A:=P(-\Delta): \V\cap \H^2(\Omega)\to \H$ where $P$ is the Leray-Helmholtz projector operator. 

Let $A_\star : \V_\star\to\V_\star$ be defined as
$$
(A_\star\v_\star, \bar\v_\star)=(\nabla\pi_{\W_h}\v_\star, \nabla\pi_{\W_h}\bar\v_\star)+ h^{-2} (\pi^\perp_{\W_h}\v_\star, \pi^\perp_{\W_h}\bar\v_\star) \quad \mbox{ for all } \quad\bar\v_\star\in \V_\star.
$$
Equivalently, one can write $A_\star=\pi_{\W_h}A_\star+\pi^\perp_{\W_h} A_\star:=A_h+\tilde A_h$ satisfying
\begin{equation}\label{Stokes-stab}
\left\{
\begin{array}{rcll}
(A_h\v_\star, \w_h)+(\nabla r_h, \w_h)&=&(\nabla \pi_{\W_h}\v_\star, \nabla \w_h)& \mbox{ for all } \quad \w_h\in\W_h, 
\\
(A_h\v_\star, \nabla q_h)+(\tilde A_h\v_\star, \nabla q_h)&=&0 & \mbox{ for all } \quad q_h\in Q_h,
\\
(\tilde A_h\v_\star, \tilde\w_h)+(\nabla r_h, \tilde\w_h)&=&h^{-2}(\pi^\perp_{\W_h}\v_\star, \tilde \w_h)& \mbox{ for all} \quad \tilde\w_h\in\tilde\W_h.
\end{array}
\right.
\end{equation}
Again, $A_\star$ is a self-adjoint, positive-definite operator. Therefore, the fractional power of $A_\star$, say $A^s_\star$, is well-defined for all $s\in \R$. Moreover, $\V^s_\star$ denotes $\V_\star$ equipped with the Hilbert norm
$$
\|\v_\star\|_{\V_\star^s}=(A_\star^{\frac{s}{2}}\v_\star, A_\star^{\frac{s}{2}}\v_\star )^{\frac{1}{2}}.
$$
The family $\{\V_\star^s \}_{s\in\R}$ is a scale of Hilbert space with respect to the real method of interpolation. 

Next we will consider a non-conforming approximation of the Leray-Helmholtz projection operator. Let $P_\star : \L^2(\Omega)\to\V_\star$ be defined as
$$
(P_\star\v, \bar\v_\star)=(\v, \bar\v_\star) \quad \mbox{ for all } \quad\bar\v_\star\in \V_\star.
$$
Equivalently, one can write $P_\star=\pi_{\W_h}P_\star+\pi^\perp_{\W_h} P_\star:=P_h+\tilde P_h$ satisfying
\begin{equation}\label{stab-Leray-proj}
\left\{
\begin{array}{rcll}
(P_h\v, \w_h)+(\nabla r_h, \w_h)&=&( \pi_{\W_h}\v,  \w_h)& \mbox{ for all } \quad \w_h\in\W_h 
\\
(P_h\v, \nabla q_h)+(\tilde P_h\v, \nabla q_h)&=&0 & \mbox{ for all } \quad q_h\in Q_h,
\\
(\tilde P_h\v, \tilde\w_h)+(\nabla r_h, \tilde\w_h)&=&(\pi_{\W_h}^\perp\v, \tilde \w_h)& \mbox{ for all } \quad \tilde\w_h\in\tilde\W_h.
\end{array}
\right.
\end{equation}

Finally, we define the stabilized Ritz projection operator onto $\V_\star$.  Let $R_\star: \H^1_0(\Omega)=\pi_{\W_h}\H^1_0\oplus \pi^\perp_{\W_h} \H^1_0(\Omega)\to \V_\star$ be defined as 
$$
(\nabla \pi_{\W_h} R_\star\v, \nabla \pi_{\W_h}\v_\star)+h^{-2} (\pi_{\W_h}^\perp R_\star \v, \pi_{\W_h}^\perp\bar\v_\star)=(\nabla\pi_{\W_h}\v, \nabla \pi_{\W_h}\v_\star)+h^{-2} (\pi^\perp_{\W_h} \v_, \pi_{\W_h}^\perp\bar\v_\star),
$$
for all $\quad\v_\star\in\V_\star$. Equivalently, one can write $R_\star=\pi_{\W_h}R_\star+\pi^\perp_{\W_h} R_\star:=R_h+\tilde R_h$ satisfying
\begin{equation}\label{stab-Ritz-proj}
\left\{
\begin{array}{rcll}
(\nabla R_h\v, \nabla \w_h)+(\nabla r_h, \w_h)&=&(\nabla\pi_{\W_h}\v,  \nabla\w_h)& \mbox{ for all } \quad \w_h\in\W_h 
\\
(R_h\v, \nabla q_h)+(\tilde R_h\v, \nabla q_h)&=&0 & \mbox{ for all } \quad q_h\in Q_h,
\\
h^{-2}(\tilde R_h\v, \tilde\w_h)+(\nabla r_h, \tilde\w_h)&=&h^{-2}(\pi_{\W_h}^\perp\v, \tilde \w_h)& \mbox{ for all } \quad \tilde\w_h\in\tilde\W_h.
\end{array}
\right.
\end{equation}

\section{Technical preliminary results} \label{sec:tech}
This section is mainly devote to some technical results concerning
equivalence between norms and inf-sup conditions in fractional-order
Sobolev spaces.
\begin{lemma}\label{lm4.1} Suppose that conditions $(\rm H1)$-$(\rm H3)$ hold. Then there exist two positive constants $c, C$ such that, for all $s\in\R$, 
\begin{equation}\label{lm4.1-lab1}
c(\|\w_h\|_{\W_h^s}+h^{-s} \|\tilde \w_h\|)\le \|\w_\star\|_{\W^s_\star}\le C (\|\w_h\|_{\W_h^s}+h^{-s} \|\tilde\w_h\|),
\end{equation}
for all $ \w_\star=\w_h+\tilde\w_h\in \W_\star$.
\end{lemma}
\begin{proof} The proof follows by observing that $(-\Delta_\star\w_\star)^s=(-\Delta \pi_{\W_h}\w_\star)^s+ h^{-s} \pi^\perp_{\W_h}\tilde\w_\star$ for all $\w_\star$. 

\end{proof}
\begin{corollary}\label{co4.2}
Suppose that conditions $(\rm H1)$-$(\rm H3)$ hold. Then there exist two positive constants $c, C$ such that, for all $s\in(-\frac{3}{2},\frac{3}{2})$, 
\begin{equation}\label{co4.2-lab1}
c(\|\w_h\|_{\tilde\H_0^s(\Omega)}+h^{-s} \|\tilde\w_h\|)\le \|\w_\star\|_{\W^s_\star}\le C (\|\w_h\|_{\tilde\H_0^s(\Omega)}+h^{-s} \|\tilde\w_h\|),
\end{equation}
for all $\w_\star=\w_h+\tilde\w_h\in \W_\star$.
\end{corollary}
\begin{proof} The proof is based on the result of \cite[Lemma 2.2]{Guermond-Pasciak_2008}:
$$c \|\w_h\|_{\tilde\H_0^s}\le \|\w_h\|_{\W^s_h}\le C \|\w_h\|_{\tilde\H_0^s}\quad \mbox{ for all }\quad \w_h\in \W_h.$$
with $s\in(-\frac{3}{2},\frac{3}{2})$.
\end{proof}
In the next lemma,  we prove the stability of the stabilized discrete Leray-Helmholtz operator $P_\star=P_h+\tilde P_h$.
\begin{lemma}\label{lm4.3} Assume that conditions $(\rm H1)$-$(\rm H3)$ are satisfied. Then there exists a positive constant C, independent of $h$, such that,  for all $s\in [0, \frac{1}{2})$,
\begin{equation}\label{stab-Ph_and_tPh}
\|P_h\v\|_{\tilde\H^s_0(\Omega)}+h^{-s}\|\tilde P_h\v\|\le C \|\v\|_{\tilde\H^s_0(\Omega)}\quad \mbox{ for all }\quad \v\in \tilde\H^s_0(\Omega),
\end{equation}
where $P_\star=P_h+\tilde P_h$ is the $\L^2(\Omega)$-orthogonal projection operator onto $\V_\star$.
\end{lemma}
\begin{proof} Let $\v\in\tilde\H^s_0(\Omega)$. Then, by the Helmholtz-Hodge decomposition, there exists $r\in H^1_{\int=0}(\Omega)$ such that 
$$
\v=P\v+\nabla r,
$$
whose variational formulation reads as:
\begin{equation}\label{lm4.2-lab1}
\left\{
\begin{array}{rclll}
(P\v, \bar\v)+(\nabla r, \bar\v)&=&(\v, \bar\v)& \mbox{ for all }&\bar \v\in \L^2(\Omega), 
\\
(\v, \nabla q)&=&0&\mbox{ for all }& \H^1_{\int =0}(\Omega),
\end{array}
\right.
\end{equation}
Note that problem (\ref{stab-Leray-proj}) is the stabilized discrete counterpart of (\ref{lm4.2-lab1}). From \cite[Chapter II, Theorem 1.1]{Girault-Raviart_1986}, we get  
\begin{equation}\label{lm4.2-lab2}
\|P_\star\v-P\v\|+\|\nabla r_h-\nabla r\|\le C (\inf_{\w_\star\in\W_\star}\|P\v-\w_\star\|+ \inf_{q_h\in Q_h}\|\nabla r-\nabla q_h\|).
\end{equation}
Using the fact that  $P_\star=P_h+\tilde P_h$ with $P_h$ and $\tilde P_h$ being $\L^2(\Omega)$-orthogonal by definition, we have 
$$
\begin{array}{rcl}
\|P_\star\v-P\v\|^2&=&\|P_h\v-P\v\|^2+\|\tilde P_h\v\|^2-2(P\v, \tilde P_h\v )
\\
&=&\|P_h\v-P\v\|^2+\|\tilde P_h\v\|^2-2(P\v-\pi_{\W_h}P\v, \tilde P_h\v )
\\
&\ge& \|P_h\v-P\v\|^2+\|\tilde P_h\v\|^2- 2\|P\v-\pi_{\W_h}P\v\|^2-\frac{1}{2} \| \tilde P_h\v\|^2
\\
&=&\|P_h\v-P\v\|^2+\frac{1}{2}\|\tilde P_h\v\|^2- 2\|P\v-\pi_{\W_h}P\v\|^2.
\end{array}
$$
Inserting this back into (\ref{lm4.2-lab2}), we obtain 
$$
\|P_h\v-P\v\|+\|\nabla r_h-\nabla r\|+\|\tilde P_h\v\|\le C (\inf_{\w_h\in\W_h}\|P\v-\w_h\|+ \inf_{q_h\in Q_h}\|\nabla r-\nabla q_h\|).
$$
From \eqref{approx-velocity} and \eqref{approx-pressure}, we find that 
\begin{equation}\label{lm4.2-lab3}
\|P_h\v-P\v\|+\|\tilde P_h \v\|\le C h^{s} \|\v\|_{\tilde\H^s_0(\Omega)},
\end{equation}
and
\begin{equation}\label{lm4.2-lab4}
h^{-s}\|\tilde P_h \v\|\le C  \|\v\|_{\tilde\H^s_0(\Omega)}. 
\end{equation}
In view of \eqref{inverseH^s-L^2andL^2-H^-s}, \eqref{stability},  \eqref{approx-velocity}  and (\ref{lm4.2-lab3}), we write 
\begin{equation}\label{lm4.2-lab5}
\begin{array}{rcl}
\|P_h\v\|_{\tilde\H^s_0(\Omega)}&\le&\|P_h\v-\pi_{\W_h} P\v\|_{\H^s_0(\Omega)}+\|\pi_{\W_h} P_h\v\|_{\tilde\H^s_0(\Omega)}
\\
&\le& C h^{-s}\|P_h\v-\pi_{\W_h} P\v\|+ C \|P\v\|_{\tilde\H^s_0(\Omega)}
\\
&\le& C h^{-s} (\|P_h\v-P\v\|+\|P\v-\pi_{\W_h} P\v\|)+C \|P\v\|_{\tilde \H^s_0(\Omega)}
\\
&\le& C \|\v\|_{\tilde\H^s_0(\Omega)}.
\end{array}
\end{equation}
In the last line we have made use of the inequality $\|P\v\|_{\H^s_0(\Omega)}\le C \|\v\|_{\tilde\H^s_0(\Omega)}$ for all $s\in[0,\frac{1}{2})$. For a proof, see \cite[Lemma 1.1]{Guermond-Pasciak_2008}. Also see  \cite[Remark 3.1]{Guermond-Pasciak_2008} for an explanation of the restriction of $s\in[0,\frac{1}{2})$.

Finally, the proof follows by combining (\ref{lm4.2-lab4}) and (\ref{lm4.2-lab5}).
\end{proof}
As a consequence of Lemma \ref{lm4.3}, we immediately obtain the following.
\begin{corollary}Assume that conditions $(\rm H1)$-$(\rm H3)$ hold. Then there exists a positive constant C, independent of $h$, such that,  for all $s\in [0, \frac{1}{2})$,
\begin{equation}\label{stab_P_star}
\|P_\star\w_\star\|_{\W^{s}_\star}\le C \|\w_\star\|_{\W^s_\star}\quad \mbox{ for all }\quad \w_\star\in \W_\star,
\end{equation}
where $P_\star=P_h+\tilde P_h$ is the $\L^2(\Omega)$-orthogonal projection operator onto $\V_\star$.
\end{corollary}
\begin{proof} Let $\w_\star\in \W_\star$ such that $\w_\star=\w_h+\tilde\w_h$. Take $\v=\w_h$ in (\ref{stab-Ph_and_tPh}) to get
\begin{equation}\label{co4.4-lab1}
\|P_h\w_h\|_{\tilde\H^s_0(\Omega)}\le C \|\w_h\|_{\tilde\H^s_0(\Omega)}.
\end{equation}
Next, select $\v=\tilde\w_h$ in (\ref{stab-Leray-proj}). Now, pick $\w_h=P_h\tilde\w_h$, $q_h=r_h$, and $\tilde\w_h=\tilde P_h\tilde\w_h$ to obtain
$$
\|P_h\tilde\w_h\|^2+\frac{1}{2}\|\tilde P_h\tilde\w_h\|^2\le \frac{1}{2}\|\tilde\w_h\|^2.
$$
In particular, we have
\begin{equation}\label{co4.4-lab2}
h^{-s}\|\tilde P_h\tilde\w_h\|^2\le h^{-s}\|\tilde\w_h\|^2.
\end{equation}
Combining  (\ref{co4.4-lab1}) and (\ref{co4.4-lab2}), we prove (\ref{stab_P_star}) by Corollary \ref{co4.2}.
\end{proof}
The following lemma sets up the equivalence between $\|\cdot\|_{\W^s_\star}$ and $\|\cdot\|_{\V^s_\star}$.
\begin{lemma}\label{lm4.5} Suppose that $(\rm H1)$-$(\rm H2)$ are satisfied. There exist two positive constants $C, c$ such that, for each $s\in (-\frac{1}{2}, 2)$,
\begin{equation}\label{W^s<V^s}
c\|\v_\star\|_{\W^s_\star}\le  \|\v_\star\|_{\V^s_\star}\quad\mbox{  for all }\quad \v_\star\in \V_\star,
\end{equation}
and, for each $s\in [-2, 2]$, 
\begin{equation}\label{V^s<W^s}
\|\v_\star\|_{\V^s_\star}\le C \|\v_\star\|_{\W^s_\star}\quad\mbox{  for all }\quad \v_\star\in \V_\star.
\end{equation}

\end{lemma}
\begin{proof} Assertion (\ref{W^s<V^s}) is proved as follows. Take $\v_\star\in \V_\star$ such that $\v_\star=\v_h+\tilde\v_h $. Let $s\in [0, 2]$.  Define $(\v, r)\in (\V\cap\H^1_0(\Omega))\times H^1_{\int=0}(\Omega)$  such that
$$
\left\{
\begin{array}{rclrl}
-\Delta \v+\nabla r&=&A_\star\v_\star& \mbox{ in } & \Omega,
\\
\nabla\cdot\v&=&0&\mbox{ in }& \Omega.
\end{array}
\right.
$$
By virtue of $(\rm H1)$, we have that $\|\Delta\v\|+\|\nabla r\|\le C \|A_\star\v_\star\|$. 
Moreover,  we have that  $(\v_h, \tilde\v_h, r_h)\in\W_h\times \tilde \W_h\times Q_h $ satisfies, for all $(\w_h, \tilde\w_h, q_h)\in \W_h\times\tilde\W_h\times Q_h$,
$$
\begin{array}{rcl}
(\nabla\v_h, \nabla\w_h)+(\nabla r_h, \w_h)&=&(A_\star\v_\star, \w_h),
\\
(\v_h, \nabla q_h)+(\tilde\v_h, \nabla q_h)&=&0,
\\
h^{-2}(\tilde \v_h, \tilde \w_h)+(\nabla r_h, \tilde \w_h)&=&(A_\star\v_\star, \tilde \w_h).
\end{array}
$$
Comparing both problems, we get the following error estimates, that can be found in \cite[Lemma 3.2 ]{Badia-Codina-Gutierrez}: 
\begin{equation}\label{lm4.5-lab1-bis}
\|\nabla (\v-\v_h)\|+h^{-1}\|\tilde\v_h\|\le C h \|A_\star\v_\star\|.
\end{equation}
Next, let us write
$$
-(\Delta_\star\v_\star, \w_\star)=(\nabla(\pi_{\W_h}\v_\star-\v), \nabla\pi_{\W_h}\w_\star)+h^{-2}(\pi^\perp_{\W_h}\v_\star, \pi^\perp_{\W_h}\w_\star)-(\Delta\v, \w_h),
$$
where we have used the fact that $-\Delta_\star\v_\star=-\Delta_h\pi_{\W_h}\v_\star- h^{-2}\pi^\perp_{\W_h}\v_\star$. Select $\w_\star=-\Delta_\star\v_\star$ to find, from \eqref{inverseH^s-L^2andL^2-H^-s} and \eqref{lm4.5-lab1-bis}, that 
$$
\begin{array}{rcl}
\|\Delta_\star\v_\star\|^2&=&\|\nabla(\v_h-\v)\| \|\nabla\pi_{\W_h}\Delta_\star\v_\star\|+h^{-2}\|\pi_{\W_h}^\perp\v_\star\| \|\pi_{\W_h}^\perp \Delta_\star\v_\star\|+\|\Delta\v\| \|\Delta_\star\v_\star\|
\\
&\le& C\| A_\star\v_\star\| \|\Delta_\star\v_\star\|+ C\|A_\star\v_\star\| \|\Delta_\star\v_\star\|.
\end{array}
$$
Thus, we obtain
\begin{equation}\label{lm4.5-lab1}
\|\Delta_\star\v_\star\|\le C \|A_\star\v_\star\|.
\end{equation}
Equivalently,
$$
\|\v_\star\|_{\W_\star^2}\le C \|\v_\star\|_{\V_\star^2}.
$$
We also have
$$
\|\v_\star\|_{\W_\star^0}\le C \|\v_\star\|_{\V_\star^0}.
$$
By interpolation, one then deduces
\begin{equation}\label{lm4.5-lab2}
\|\v_\star\|_{\W_\star^s}\le C \|\v_\star\|_{\V_\star^s}.
\end{equation}
for all $s\in[0,2]$. 

Let $s\in (-\frac{1}{2}, 0]$.  Then
$$
\|\v_\star\|_{\W^s_\star}=\sup_{\w_\star\in \W^{-s}_\star\backslash \{\boldsymbol{0}\}}\frac{(\v_\star, \w_\star)}{\|\w_\star\|_{\W^{-s}_\star}}=\sup_{\w_\star\in \W^{-s}_\star\backslash \{\boldsymbol{0}\}}\frac{(\v_\star, P_\star\w_\star)}{\|\w_\star\|_{\W^{-s}_\star}}.
$$
In view of \eqref{stab_P_star}, we have
$$
\begin{array}{rcl}
\|\v_\star\|_{\W^s_\star}&\le&\displaystyle C \sup_{\w_\star\in  \W^{-s}_\star\backslash \{\boldsymbol{0}\}}\frac{(\v_\star, P_\star\w_\star)}{\|P_\star\w_\star\|_{\W^{-s}_\star}}\le C \sup_{\w_\star\in  \V^{-s}_\star\backslash \{\boldsymbol{0}\}}\frac{(\v_\star,\w_\star)}{\|\w_\star\|_{\V^{-s}_\star}}
\\
&\le&\displaystyle C \|\v_\star\|_{\V_\star^s}  \sup_{\w_\star\in \V^{-s}_\star\backslash \{\boldsymbol{0}\}}\frac{\|\w_\star\|_{\V^{-s}_\star}}{\|\w_\star\|_{\W^{-s}_\star}}\le C \|\v_\star\|_{\V_\star^s}.
\end{array}
$$
Assertion (\ref{V^s<W^s}) is proved as follows. Pick $\v_\star	\in\V_\star$. Let $s\in  [0, 2] $.  It should be first noted that $A_\star R_\star \w_\star = P_\star\Delta_\star\w_\star$ for all $\w_\star\in\W_\star$. Then, it follows that  
\begin{equation}\label{lm4.5-lab3}
\|R_\star\w_\star\|_{\V^2_\star}=\|A_\star R_\star\w_\star\|=\|P_\star\Delta_\star\w_\star\|\le \|\Delta_\star\w_\star\|=\|\w_\star\|_{\W^2_\star}.
\end{equation}
Next, we have that 
\begin{equation}\label{lm4.5-lab4}
\|R_\star\w_\star\|\le C \|\w_\star\|
\end{equation}
holds for all $\w_\star\in \W_\star$, i.e., $\|R_\star\w_\star\|_{\V^0_\star}\le C \|\w_h\|_{\W^0_h}$  for all $\w_\star\in \W_\star$. Indeed, let $R^T_\star: \V_\star\to\W_\star $ the adjoint operator of $R_\star$, namely $R^T_\star=-\Delta_\star A_\star^{-1}$, since $R_\star=- A_\star^{-1} P_\star \Delta_\star$. It is clear that if 
\begin{equation}\label{lm4.5-lab5}
\|R^T_\star \v_\star\|\le C \|\v_\star\|
\end{equation}
holds for all $\v_\star\in\V_\star$, then (\ref{lm4.5-lab4}) is true. But (\ref{lm4.5-lab5}) is a consequence of (\ref{lm4.5-lab1}). By interpolation between \eqref{lm4.5-lab3} and \eqref{lm4.5-lab4} for $\w_\star=\v_\star$, we obtain  $\|\v_\star\|_{\V^s_\star}\le C \|\v_\star\|_{\W^s_\star} $ for all $\v_\star\in \V_\star$ and $s\in[0,2]$. 

Let $s\in [-2, 0]$. From (\ref{lm4.5-lab2}), we have
$$
\|\v_\star\|_{\V^s_\star}=\sup_{\v_\star\in\V_\star\backslash\{\boldsymbol 0\}}\frac{(\v_\star, \bar\v_\star)}{\|\bar\v_\star\|_{\V^{-s}_\star}}\le C \sup_{\v_\star\in\V_\star\backslash\{\boldsymbol 0\}}\frac{(\v_\star, \bar\v_\star)}{\|\bar\v_\star\|_{\W^{-s}_\star}}\le\|\v_\star\|_{\W^s_\star}.
$$ 

It completes the proof.
\end{proof}
As a corollary to Lemma \ref{lm4.5}, we have the following inequality whose proof needs Corollary \ref{co4.2}. 
\begin{corollary}\label{co4.6}  Suppose that conditions $(\rm H1)$-$(\rm H2)$ are satisfied. Then there exist two positive constants $C, c$ such that, for each $s\in (-\frac{1}{2}, \frac{3}{2})$,
\begin{equation}\label{co4.6-lab1}
c(\|\v_h\|_{\tilde\H^s_0(\Omega)}+h^{-s}\|\tilde \v_h\|)\le  \|\v_\star\|_{\V^s_\star} \quad \mbox{ for all }\quad\v_\star\in\V_\star,
\end{equation}
and, for each $s\in (-\frac{3}{2}, \frac{3}{2})$,  
\begin{equation}\label{co4.6-lab2}
 \|\v_\star\|_{\V^s_\star}\le C(\|\v_h\|_{\tilde\H^s_0(\Omega)}+h^{-s}\|\tilde \v_h\|)\quad \mbox{ for all }\quad\v_\star\in\V_\star.
\end{equation}
\end{corollary}
We are now concerned with the proof of an inf-sup condition in the framework of fractional Sobolev spaces.
\begin{lemma} Under conditions $(\rm H1)$-$(\rm H3)$, it follows that, for $s\in [0,1]$,
$$
\sup_{\w_\star \in  \W_\star\backslash \{ \boldsymbol 0\}} \frac{\left({\nabla q_h},{\w_\star}\right)}{\|\w_\star \|_{\W^{1-s}_\star}}  \gtrsim \|q_h\|_{H^s(\Omega)},
$$
\end{lemma}
\begin{proof} Let $q_h\in Q_h$. From \eqref{stability} and \eqref{approx-velocity}, we have 
$$
\begin{array}{rcl}
\displaystyle
\sup_{\w_h\in \W_h\backslash \{ \boldsymbol 0\}} \frac{(\nabla q_h, \w_h)}{\|\w_h\|_{ \tilde\H^{1-s}_0(\Omega)}} &\ge& \displaystyle \sup_{\w\in \tilde\H_0^{1-s}(\Omega)\backslash  \{ \boldsymbol 0\}} \frac{(\nabla q_h, \pi_{\W_h}\w)}{\|\pi_{\W_h}\w\|_{\tilde\H^{1-s}_0(\Omega)}}\ge C_{\rm st}^{-1} \sup_{\w\in \tilde\H_0^{1-s}(\Omega)\backslash  \{ \boldsymbol 0\}} \frac{(\nabla q_h, \pi_{\W_h}\w)}{\|\w\|_{\tilde\H^{1-s}_0(\Omega)}}
\\
&\ge& C_{\rm st}^{-1} \displaystyle\sup_{\w\in \tilde\H^{1-s}_0\backslash  \{ \boldsymbol 0\}} \frac{(\nabla q_h, \w)}{\|\w\|_{\tilde\H^{1-s}_0(\Omega)}}- C_{\rm st}^{-1}\sup_{\w\in \tilde\H^s_0\backslash  \{ \boldsymbol 0\}} \frac{(\nabla q_h, \pi_{\W_h}^\perp\w)}{\|\pi_{\W_h}\w\|_{\tilde\H^{1-s}_0(\Omega)}}
\\
&\ge&\displaystyle  C_{\rm st}^{-1} \|q_h \|_{H^s(\Omega)} - C_{\rm st}^{-1} C_{\rm int} h^{1-s} \|\nabla q_h\|.
\end{array}
$$
Moreover, we have, by \eqref{inverseH^s-L^2andL^2-H^-s}, 
$$
\sup_{\w_h \in  \W_h\backslash  \{ \boldsymbol 0\}} \frac{\left({\nabla q_h},{\w_h}\right)}{\|\w_h \|_{\tilde\H^{1-s}_0(\Omega)}} \ge
\frac{\|\pi_{\W_h}\nabla q_h\|^2}{\|\pi_{\W_h}\nabla q_h\|_{\tilde\H^{1-s}_0(\Omega)}} \ge C_{\rm inv }^{-1} h^{1-s} \| \pi_{\W_h}\nabla q_h \|.
$$
Combining the previous two estimates, we get
$$
\sup_{\w_h \in \W_h\backslash \{ \boldsymbol 0\}} \frac{\left({\nabla q_h},{\w_h}\right)}{\|\w_h \|_{\tilde\H^{1-s}_0(\Omega)}} \ge C_1 \|q_h \|_{H^s(\Omega)} - C_2 h^{1-s} \|\pi_{\W_h}^\perp\nabla q_h\|,
$$
where $C_1^{-1}= C_2^{-1} C_{\rm int}$ and $C_2^{-1}=(C_{\rm st} +C_{\rm int} C_{\rm inv})$. 
Since $\pi^\perp_{\W_h}(\nabla Q_h) \subset \tilde \W_h$, we find that 
$$
\sup_{\tilde\w_h \in  \tilde \W_h\backslash  \{ \boldsymbol 0\}} \frac{({\nabla q_h},{\tilde \w_h})}{h^{s-1}\| \tilde\w_h \|} \ge  h^{1-s} \|\pi_{\W_h}^\perp\nabla q_h\|.
$$
At this point, we have proved that 
$$
\sup_{\w_h \in  \W_h\backslash  \{ \boldsymbol 0\}} \frac{\left({\nabla q_h},{\w_h}\right)}{\|\w_h \|_{\tilde\H^{1-s}_0(\Omega)}} + C_2\sup_{\tilde\w_h \in  \tilde \W_h\backslash  \{ \boldsymbol 0\}} \frac{\left({\nabla q_h},{\tilde \w_h}\right)}{h^{s-1}\| \tilde\w_h \|} \ge C_1 \|q_h\|_{H^s(\Omega)}.
$$
Finally, observe, by \eqref{lm4.1-lab1}, that  
$$
\sup_{\w_\star \in  \W_\star\backslash \{ \boldsymbol 0\}} \frac{\left({\nabla q_h},{\w_\star}\right)}{\|\w_\star \|_{\W^{1-s}_\star}}\ge  \sup_{\w_h \in  \W_h\backslash \{ \boldsymbol 0\}} \frac{\left({\nabla q_h},{\w_h}\right)}{\|\tilde\w_h \|_{\tilde\H_0^{1-s}(\Omega)}}
$$
and
$$
\sup_{\w_\star \in  \W_\star\backslash \{ \boldsymbol 0\}} \frac{\left({\nabla q_h},{\w_\star}\right)}{\|\w_\star \|_{\W^{1-s}_\star}}\ge  \sup_{\tilde\w_h\in \tilde \W_h\backslash \{ \boldsymbol 0\}} \frac{\left(\nabla q_h,\tilde\w_h\right)}{h^{s-1}\|\w_h \|}.
$$
Then it follows that 
$$
(1+C_2)\sup_{\w_\star \in  \W_\star\backslash \{ \boldsymbol 0\}} \frac{\left({\nabla q_h},{\w_\star}\right)}{\|\w_\star \|_{\W^{1-s}_\star}}\ge C_1\|q_h\|_{H^s(\Omega)}
$$
or equivalently, for any $q_h \in Q_h$, there exists an element $\w_\star \in \W_\star$ with norm $\|\w_\star\|_{\W^{1-s}_\star} = 1$, such that $\left( \nabla q_h , v_\star \right) \ge C \|q_h \|_{H^s(\Omega)}$. It easily proves the lemma.
\end{proof}
\section{A priori energy estimates}\label{sec:energy}
In this section we derive a series of a priori energy estimates resulting in that the sequence of the approximate solutions $(\u_h, p_h)$ to scheme (\ref{eq:GalerkinNS}) approaches to a weak solution and a suitable weak solution in Section \ref{weaksol}.   
\begin{theorem} Under assumptions $\rm (H1)$, $\rm (H2)$, and $\rm (H4)$, there is a positive constant $C$, independent of $h$, such that 
\begin{equation}\label{energy-estimate}
\|\u_h\|_{L^\infty(0,T; \L^2(\Omega))\cap L^2(0,T; \H^1_0(\Omega))}+\|\tilde\u_h\|_{L^\infty(0,T; \L^2(\Omega))\cap \tau^{-1}L^2(0,T; \L^2(\Omega))}\le C.
\end{equation}
\end{theorem}
\begin{proof} Take $\v_h=\u_h$ in (\ref{eq:Galerkin-u}), $\tilde \v_h=\tilde\u_h$ in (\ref{eq:Sub-Scale}) and $q_h=p_h$ in (\ref{eq:Galerkin-p}) to get
$$
\frac{1}{2} \frac{\rm d}{\dt} \left( \| \u_h\|^2 +\|\tilde\u_h\|^2
\right) + \nu \| \nabla \u_h \|^2+ \tau^{-1} \| \tilde\u_h \|^2 =
(\f_h , \u_h).
$$
Next we estimate the term $(\f_h , \u_h)$. Thus, we have
$$
(\f_h , \u_h )\le \|\f\|_{\H^{-1}(\Omega)} \|\nabla\u_h\|\le \frac{ 1}{2\nu} \|\f\|^2_{\H^{-1}(\Omega)}+\frac{\nu}{2}\|\nabla\u_h\|^2.
$$
Therefore, we obtain
$$
\frac{\rm d}{\dt} \left( \| \u_h\|^2 + \|\tilde\u_h\|^2 \right)  + \nu \| \nabla \u_h \|^2 + {\tau^{-1}} \| \tilde\u_h \|^2 \le \frac{1}{\nu}\| \f \|^2_{\H^{-1}(\Omega)},
$$
which, integrated over $(0,t)$, leads to the desired result.
\end{proof}
\begin{remark}From (\ref{energy-estimate}), we have 
\begin{equation}\label{est_LrLk}
\|\u_h\|_{L^r(0,T; \H^{2/r}(\Omega))}+\|\u_h\|_{L^r(0,T;\L^k(\Omega))}\le C,
\end{equation}
where $\frac{3}{k}+\frac{2}{r}=\frac{3}{2}$ with $r\in[2,\infty]$ and $k\in[2, 6]$. The proof is based on the interpolation inequality between $\L^2(\Omega)$ and $\H^1(\Omega)$, i.e., $\|\w\|_{\H^{\frac{2}{r}}(\Omega)}\le C \|\w\|^{1-\frac{2}{r}}_{\L^2(\Omega)} \|\w\|^{\frac{2}{r}}_{\H^1(\Omega)}$ for $r\in[2,\infty]$, and the Sobolev embedding $ \H^\frac{2}{r}(\Omega) \hookrightarrow \L^k(\Omega)$ for $\frac{1}{k}=\frac{1}{2}-\frac{2}{3r}$ and $r>\frac{4}{3}$.
\end{remark}

Consider $\u_\star=\u_h+\tilde\u_h$. Then problem (\ref{eq:GalerkinNS}) reads as follows. Find $\u_\star\in H^1(0,T; \W_\star)$ and $p_h\in L^2(0,T; Q_h)$ such that
\begin{equation}\label{Sec5-Lab1}
(\partial_t \u_\star,\w_\star) +b_\star(\u_\star,\u_\star, \w_\star)
-\nu(\Delta_\star\u_\star,\w_\star)-(\nabla p_h, \w_\star)= (\f_h,\pi_{\W_h}(\w_\star)),
\end{equation}
for all $\w_\star\in \W_\star$, with the initial condition $\u_\star(0)=\u_{0h} + \tilde\u_{0h}$, where 
$$
b_\star(\u_\star,\v_\star, \w_\star)=b(\pi_{\W_h}\u_\star,\pi_{\W_h}\u_\star,\w_\star)-b (\pi_{\W_h}\u_\star, \pi_{\W_h}\w_\star,\pi^\perp_{\W_h} \u_\star)+ \tau^{-1}_\infty (\pi_{\W_h}^\perp\u_\star,\pi_{\W_h}^\perp \w_\star )
$$
with $\tau_\infty^{-1}=C_c \frac{\|\pi_{\W_h}\u_\star\|_{L^\infty(\Omega)}}{h}$.
We now define the operator $ \mathcal {N}_\star:\W_\star\times \W_\star\to \W_\star^{-1}$ via duality by the formula
$$
\langle \mathcal {N}_\star(\u_\star, \v_\star), \w_\star\rangle=b_\star(\u_\star,\v_\star, \w_\star). 
$$
\begin{lemma} Suppose that  assumptions $(\rm H1)$-$(\rm H3)$ hold. Let $q$ be as in $\rm(H4)$, but $q\in[1,\frac{3}{2})$, and let $s$ be as in $\rm (C)$, i.e, $s\in[\frac{1}{2},\frac{3}{2})$.  Then it follows that 
$$\|{\mathcal N}_\star(\u_\star, \u_\star)\|_{ \W^{-s}_\star}\le C \|\pi_{\W_h}\u_\star\|_{\L^k(\Omega)} (\|\nabla\pi_{\W_h} \u_\star\| +\tau^{-\frac{1}{2}} \|\pi_{\W_h}^\perp\u_\star\| )$$
with $\frac{1}{k}+\frac{1}{2}=\frac{1}{q}$.
\end{lemma}
\begin{proof} By definition of $b(\cdot, \cdot, \cdot)$, we have
$$
b(\pi_{\W_h}\u_\star, \pi_{\W_h}\u_\star, \w_\star )=\langle\mathcal{N}(\pi_{\W_h}\u_\star,\pi_{\W_h}\u_\star), \w_\star\rangle\le \|{\mathcal N}(\pi_{\W_h}\u_\star, \pi_{\W_h}\u_\star)\|_{ \W^{-s}_\star} \|\w_\star\|_{\W^s_\star},
$$
which, combined with \eqref{lm4.1-lab1}, gives  
$$
\|{\mathcal N}(\pi_{\W_h}\u_\star, \pi_{\W_h}\u_\star)\|_{ \W^{-s}_\star}\le C \|{\mathcal N}(\pi_{\W_h}\u_\star, \pi_{\W_h}\u_\star)\|_{\tilde\H^{-s}_0(\Omega)}+ C h^{s} \|{\mathcal N}(\pi_{\W_h}\u_\star, \pi_{\W_h}\u_\star)\|.$$ 
From the continuous embedding $\tilde\H^s_0(\Omega) \hookrightarrow \L^{q'}(\Omega)$, with $\frac{1}{q'}+\frac{1}{q}=1$, we bound    
$$
\|{\mathcal N}(\pi_{\W_h}\u_\star, \pi_{\W_h}(\u_\star))\|_{ \W^{-s}_\star}\le C \|{\mathcal N}_h(\pi_{\W_h}\u_\star, \pi_{\W_h}\u_\star)\|_{ \L^q(\Omega)}+h^{s} \|{\mathcal N}(\pi_{\W_h}\u_\star, \pi_{\W_h}\u_\star)\|.$$ 
Next, using  (\ref{inverseL^inf-L^k}), the definition of $s$, and the relation $\frac{1}{k}+\frac{1}{2}=\frac{1}{q}$ gives
$$
\|{\mathcal N}(\pi_{\W_h}\u_\star, \pi_{\W_h}\u_\star)\|_{ \L^q(\Omega)}\le \|\nabla\pi_{\W_h}\u_\star\| \|\pi_{\W_h}\u_\star\|_{\L^k(\Omega)},
$$
and 
$$
h^{s} \|{\mathcal N}_h(\pi_{\W_h}\u_\star, \pi_{\W_h}\u_\star)\| \le  C h^{s} \|\pi_{\W_h}\u_\star\|_{\L^\infty(\Omega)} \|\nabla\pi_{\W_h}\u_\star\| \le C \|\nabla \pi_{\W_h}\u_\star\| \|\pi_{\W_h}\u_\star\|_{\L^k(\Omega)},
$$
which imply that 
$$
b(\pi_{\W_h}\u_\star, \pi_{\W_h}\u_\star, \w_\star )\le  \|\nabla\pi_{\W_h}\u_\star\| \|\pi_{\W_h}\u_\star\|_{\L^k(\Omega)} \|\w_\star\|_{\W^s_\star}.
$$

Next, for $s\in[\frac{1}{2},1]$, we write
$$
\begin{array}{rcl}
b(\pi_{\W_h}\u_\star, \pi_{\W_h}\w_\star, \pi^\perp_{\W_h}\u_\star)&=&((\pi_{\W_h}\u_\star\cdot\nabla) \pi_{\W_h}\w_\star, \pi_{\W_h}^\perp\u_\star)
\\
&&+\frac{1}{2} (\nabla\cdot\pi_{\W_h}\u_\star\, \pi_{\W_h}\w_\star ,\pi_{\W_h}^\perp\u_\star)
\\
&\le&\|\pi_{\W_h}\u_\star\|_{\L^\infty(\Omega)}  \|\nabla\pi_{\W_h}\w_\star\| \|\pi_{\W_h}^\perp\u_\star\|
\\
&&+\frac{1}{2}\|\nabla\cdot\pi_{\W_h}\u_\star\|_{\L^k(\Omega)}\|\pi_{\W_h}\w_\star\|_{\L^\infty(\Omega)}\|\pi_{\W_h}^\perp\u_\star\| 
\\
&\le& C h^{-\frac{3}{k}} \|\pi_{\W_h}\u_\star\|_{\L^k(\Omega)} h^{-1+s}\|\pi_{\W_h}\w_\star\|_{\tilde\H^s_0(\Omega)}\|\pi_{\W_h}^\perp\u_\star\|
\\
&&+C h^{-1} \|\pi_{\W_h}\u_\star\|_{\L^k(\Omega)}  \|\pi_{\W_h}\w_\star\|_{\L^{q'}(\Omega)} \|\pi_{\W_h}^\perp\u_\star\|  
\\
&\le& C \|\pi_{\W_h}\u_\star\|_{\L^k(\Omega)}\tau^{-\frac{1}{2}}\|\pi_{\W_h}^\perp\u_\star\|  \|\pi_{\W_h}\w_\star\|_{\tilde\H^s_0(\Omega)}
\\
&\le& \|\pi_{\W_h}\u_\star\|_{\L^k(\Omega)}\tau^{-\frac{1}{2}}\|\pi_{\W_h}^\perp\u_\star\| \|\w_\star\|_{\W^s_\star},
\end{array}
$$
where we have utilized \eqref{inverseL^inf-L^k} and \eqref{inverseH^1-H^s}, the relation $\frac{1}{k}+\frac{1}{2}=\frac{1}{q}$ and the definition of $s$. Moreover, we have also utilized $\|\pi_{\W_h}\u_\star\|_{\L^{q'}(\Omega)}\le  C\|\pi_{\W_h}\u_\star\|_{\tilde\H^s_0(\Omega)}$,  $q'\in[3,\infty]$, and the relation \eqref{co4.2-lab1}. The estimate for $s\in(1,\frac{3}{2})$ follows readily from applying the above arguments and the continuous embedding $\tilde\H^s_0(\Omega) \hookrightarrow \H^{1}_0(\Omega)$.    

Finally,  
$$
\begin{array}{rcl}
\tau_\infty^{-1} (\pi_{\W_h}^\perp \u_\star , \pi_{\W_h}^\perp\w_\star )&\le& \tau_\infty^{-1} \|\pi_{\W_h}^\perp\u_\star\| \|\pi_{\W_h}^\perp\u_\star\|
\\
&=&C_c h^{-1}\|\pi_{\W_h}\u_\star\|_{\L^\infty(\Omega)} \|\pi_{\W_h}^\perp\u_\star\| \|\pi_{\W_h}^\perp\w_\star\|
\\
&\le& C \|\pi_{\W_h}\u_\star\|_{\L^k(\Omega)} \|\pi_{\W_h}^\perp\u_\star\| h^{-s} \|\pi_{\W_h}^\perp\w_\star\|.
\end{array}
$$
A duality argument shows then that  
$$ \|{\mathcal N}_\star(\u_\star, \u_\star)\|_{\W^{-s}_\star}\le C (\tau^{-\frac{1}{2}}\|\pi_{\W_h}^\perp\u_\star\|+\|\nabla\pi_{\W_h}\u_\star\|)\|\pi_{\W_h}\u_\star\|_{\L^k(\Omega)}.$$

\end{proof}
\begin{corollary}Suppose that  assumptions $(\rm H1)$-$(\rm H3)$ hold. Let $p$ be as in $({\rm H4})$, but $p\in[1,2)$, and  $s$ be as in $(\rm C)$, but $s\in[\frac{1}{2}, \frac{3}{2})$. Then there exists $C>0$, independent of $h$, such that 
$$
\|{\mathcal N}_\star(\u_\star, \u_\star)\|_{ L^{p}(0,T; \W^{-s}_\star)}\le C
$$
\end{corollary}
\begin{proof} The proof follows by using the $L^{r}(0,T; \L^k(\Omega))$ estimate stated in (\ref{est_LrLk}).
\end{proof}
\begin{lemma}\label{lm5.5} Suppose that $(\rm H1)$-$(\rm H3)$  hold. Let $s$ and $\bar r$ be as in $(\rm C)$, but  $s\in[\frac{1}{2}, \frac{3}{2})$ and $\bar r\in[0,\frac{1}{2})$. Then there exists  $C>0$, independent of $h$, such that
\begin{equation}\label{lm5.5-lab1}
\|{\mathcal N}_\star(\u_\star, \u_\star)\|_{ H^{-r}(\R;\W^{-s}_\star)}\le C, 
\end{equation}
with $r>\bar r$.
\end{lemma} 
\begin{proof} Let ${\mathfrak N}_\star$ be the extension of ${\mathcal N}_\star$ by zero off $[0,T]$. We have, by the Hausdorff-Young inequality, that 
$$
\begin{array}{rcl}
\|{\mathfrak N}_\star(\u_\star,  \u_\star )\|_{ H^{-r}(\R; \W^{-s}_\star)}&=&\displaystyle\int_{\R} (1+|\xi|)^{-2r}\|{\mathcal F}{\mathfrak N}_\star(\u_\star, \u_\star)\|^2_{ \W^{-s}_\star }d\xi
\\
&\le&\|(1+|\xi|)^{-2r}\|_{L^{\frac{1}{2\bar r}}(\R)} \|{\mathcal F}{\mathfrak N}_\star(\u_\star, \u_\star)\|^2_{ L^{p'}(\R; \W^{-s}_\star)}.
\\
&\le&C \|{\mathfrak N}_\star(\u_\star, \u_\star)\|_{ L^{p}(0,T; \W^{-s}_\star)}.
\end{array}
$$
Observe that $\|(1+|\xi|)^{-2r}\|_{L^{\frac{1}{2\bar r}}(\R)}=\|(1+|\xi|)^{-\frac{r}{\bar r}}\|_{L^1(\R)}< \infty$ if $ \frac{r}{\bar r}>1$. In particular, it holds for $r>\bar r$. It completes the proof.
\end{proof}
Before proceeding with a priori energy estimates for $\partial_t\u_h$ and $\partial_t\tilde \u_h$, let us write (\ref{Sec5-Lab1})  as a nonlinear heat equation
$$
\partial_t\u_\star+\nu A_\star\u_\star =P_\star\g,
$$
where
$$
\g:=\f_h-\mathcal{N}_\star(\u_\star, \u_\star).
$$
\begin{theorem}\label{thm5.6} Suppose that assumptions $(\rm H1)$-$(\rm H4)$ hold. Let $s$ and $\bar r$ be as in $(\rm C)$, but  $s\in[\frac{1}{2}, \frac{3}{2})$ and $\bar r\in[0,\frac{1}{2})$. Then there exists $C>0$, independent of $h$, such that 
\begin{equation}\label{thm5.6-lab1}
\|\partial_t\u_\star\|_{H^{\beta-1}(0,T; \W^{-\alpha}_\star)}+\|\u_\star\|_{H^{\beta}(0,T; \W^{-\alpha}_\star)}\le C,
\end{equation}
for all  $\alpha $ satisfying $0\le\alpha\le s\le 1+2\alpha<2$, and for all $\beta$ satisfying $\beta<\bar\beta:=\frac{1+\alpha}{1+s}(\frac{s}{2}+\frac{1}{4})$. Furthermore, 
\begin{equation}\label{thm5.6-lab2}
\|\Delta_\star\u_\star\|_{H^{-r}(0,T; \W^{-s}_\star)}\le C
\end{equation}
for $r$ satisfying $r>\bar r=\frac{3}{4}-\frac{s}{2}=\frac{1}{p}-\frac{1}{2}$.
\end{theorem}
\begin{proof} Let  $\boldsymbol{\mathfrak{g}}$ be the extension of $\g$ by zero off $[0,T]$. Let us define 
$$
\boldsymbol{\mathfrak{u}}_\star=\left\{
\begin{array}{lcl}
\boldsymbol{0} &\mbox{ for } &t\in(-\infty, -1],
\\
 (t+1)\u_{0\star} &\mbox{ for }&t\in[-1,0],
 \\
 \u_\star&\mbox{ for }&t\in[0,T+1],
 \\
 \boldsymbol{0}&\mbox{ for }& t\in[T+1,\infty).
\end{array}
\right.
$$ 
Fix a cutoff function $\varphi\in \mathcal{D}(\Omega)$ which equals $1$ on $[0,T]$ and vanishes on $(-1, T+1)$. Define  
$\tilde {\boldsymbol{\mathfrak{u}}}_\star=\varphi \boldsymbol{\mathfrak{u}}_\star$, and 
$$
\tilde{\boldsymbol{\mathfrak{g}}}=\left\{
\begin{array}{ll}
(1+t)\varphi' \u_{0\star}+ \varphi\u_{0\star}+\nu(1+t)\varphi A_\star\u_{0\star}, &\mbox{ for }t\in(-1,0), 
\\
\varphi \boldsymbol{\mathfrak{g}}+\varphi' \boldsymbol{\mathfrak{u}}_\star, & \mbox{ otherwise}. 
\end{array} 
\right.
$$
Write
$$
\partial_t\tilde{\boldsymbol{\boldsymbol{\mathfrak{u}}}}_\star+\nu A_\star\tilde{\boldsymbol{\boldsymbol{\mathfrak{u}}}}_\star =P_\star\tilde{\boldsymbol{\mathfrak{g}}}\quad\mbox{ in }\quad {\boldsymbol{\mathcal{S}}}'(\R, \V_\star).
$$
Applying the Fourier transform, we get
\begin{equation}\label{thm5.6-lab3}
2  \pi i \xi {\mathcal F}\tilde{\boldsymbol{\mathfrak{u}}}_\star+\nu A_\star{\mathcal F}\tilde{\boldsymbol{\mathfrak{u}}}_\star=P_\star{\mathcal F}\tilde{\boldsymbol{\mathfrak{g}}}.
\end{equation}
Let $\alpha\in\R^+$. Multiply \eqref{thm5.6-lab3} by the complex conjugate of $A^{-\alpha}_\star {\mathcal F}\boldsymbol{\boldsymbol{\mathfrak{u}}}_\star$ and take the imaginary part to get
$$
\begin{array}{rcl}
2\pi |\xi|\|{\mathcal F}\tilde{\boldsymbol{\boldsymbol{\mathfrak{u}}}}_\star\|^2_{\V^{-\alpha}_\star}&\le& \|{\mathcal F} \tilde{\boldsymbol{\mathfrak{g}}}\|_{\W^{-s}_\star} \|A^{-\alpha}_\star {\mathcal F}\tilde{\boldsymbol{\boldsymbol{\mathfrak{u}}}}_\star\|_{\W^s_\star}
\\
&\le &\|{\mathcal F} \tilde{\boldsymbol{\mathfrak{g}}}\|_{\W^{-s}_\star} \|{\mathcal F}\tilde{\boldsymbol{\boldsymbol{\mathfrak{u}}}}_\star\|_{\V^{s- 2\alpha}_\star},
\end{array}
$$
where we have applied \eqref{W^s<V^s} for $s\in[0,2)$. For $\alpha\le s\le 1+2\alpha$, i.e. $-\alpha\le s-2\alpha\le 1$, we interpolate to get
$$
\|{\mathcal F}\tilde{\boldsymbol{\boldsymbol{\mathfrak{u}}}}_\star\|_{\V^{s-2\alpha}_\star}\le \| {\mathcal F}\tilde{\boldsymbol{\boldsymbol{\mathfrak{u}}}}_\star\|_{\V^{-\alpha}_\star}^\gamma\|{\mathcal F}\tilde{\boldsymbol{\boldsymbol{\mathfrak{u}}}}_\star\|_{\V_\star}^{1-\gamma},
$$
where $\gamma=\frac{2\alpha+1-s}{1+\alpha}$. Therefore,
$$
2\pi |\xi|\|{\mathcal F}\tilde{\boldsymbol{\boldsymbol{\mathfrak{u}}}}_\star\|^{2-\gamma}_{\V^{-\alpha}_\star}\le C\|{\mathcal F} \tilde{\boldsymbol{\mathfrak{g}}}\|_{\W^{-s}_\star} \|{\mathcal F}\boldsymbol{\boldsymbol{\mathfrak{u}}}_\star\|_{\V_\star}^{1-\gamma},
$$
and hence
$$
|\xi|^{\frac{2}{2-\gamma}-\mu} \|{\mathcal F}\tilde{\boldsymbol{\mathfrak{u}}}_\star\|_{\V^{-\alpha}_\star}^2\le C (1+|\xi|)^{-\mu}\|{\mathcal F} \tilde{\boldsymbol{\mathfrak{g}}}\|_{\W^{-s}_\star}^{\frac{2}{2-\gamma}}\|{\mathcal F}\tilde{\boldsymbol{\mathfrak{u}}}_\star\|_{\V_\star}^{\frac{2(1-\gamma)}{2-\gamma}},
$$
where $\mu=\frac{2r}{2-\gamma}$. Integrating over $\R$  and using Hölder's inequality and Plancherel's equality gives 
$$
\int_{\R} |\xi|^{\frac{2}{2-\gamma}-\mu} \|{\mathcal F}\tilde{\boldsymbol{\mathfrak{u}}}_\star\|^2_{\V^{-\alpha}_\star} d\xi\le C \|\tilde{\boldsymbol{\mathfrak{g}}}\|_{H^{-r}(\R;\W^{-s}_\star)}^{\frac{2}{2-\gamma}}  \|\tilde{\boldsymbol{\mathfrak{u}}}_\star\|_{L^2(\R;\V_\star)}^{\frac{2(1-\gamma)}{2-\gamma}},
$$
which implies that 
\begin{equation}\label{thm5.6-lab4}
\int_{\R}|\xi|^{2\beta} \|{\mathcal F}\tilde{\boldsymbol{\mathfrak{u}}}_\star\|^2_{\V^{-\alpha}_\star}\le C\|\tilde{\boldsymbol{\mathfrak{g}}}\|_{H^{-r}(0,T;\W^{-s}_\star)}^{\frac{2}{2-\gamma}}\|\tilde{\boldsymbol{\mathfrak{u}}}_\star\|_{L^2(0,T;\V_\star))}^{\frac{2(1-\gamma)}{2-\gamma}},
\end{equation}
for $\beta<\bar\beta$ with $\bar\beta:=\frac{1+\alpha}{1+s}(1-\bar r)$ coming from the definition of $\gamma$, $\mu$, $\bar r$, and $\alpha\le s \le 1+2\alpha$. Next observe that we have, from  \eqref{rm3.1-lab1} and \eqref{lm5.5-lab1} for $s\in[0,\frac{3}{2})$, 
\begin{equation}\label{thm5.6-lab5}
\|\tilde{\boldsymbol{\mathfrak{g}}}\|_{H^{-r}(\R;\W^{-s}_\star)}\le C. 
\end{equation}
Inserting \eqref{energy-estimate} and \eqref{thm5.6-lab5} into \eqref{thm5.6-lab4}, we arrive at
$$
\int_{\R}|\xi|^{2\beta} \|{\mathcal F}\tilde{\boldsymbol{\mathfrak{u}}}_\star\|^2_{\V^{-\alpha}_\star}\le C.
$$
For $\beta\ge0$, we write
$$
\begin{array}{rcl} 
\displaystyle
\int_{\R}(1+|\xi|)^{2\beta} \|{\mathcal F}\tilde{\boldsymbol{\mathfrak{u}}}_\star\|^2_{\V^{-\alpha}_\star} d\xi&=&\displaystyle\int_{|\xi|\le1} (1+|\xi|)^{2\beta} \|{\mathcal F}\tilde{\boldsymbol{\mathfrak{u}}}_\star\|^2_{\V^{-\alpha}_\star} d\xi + \int_{|\xi|>1}(1+|\xi|)^{2\beta} \|{\mathcal F}\boldsymbol{\mathfrak{u}}_\star\|^2_{\V^{-\alpha}_\star} d\xi
\\
&\le&\displaystyle C \int_{|\xi|\le1}  \|{\mathcal F}\tilde{\boldsymbol{\mathfrak{u}}}_\star\|^2_{\V^{-\alpha}_\star} d\xi + C  \int_{|\xi|>1} |\xi|^{2\beta}\|{\mathcal F}\tilde{\boldsymbol{\mathfrak{u}}}_\star\|^2_{\V^{-\alpha}_\star} d\xi.
\\
&\le&\displaystyle C  \int_\R\| \tilde{\boldsymbol{\mathfrak{u}}}_\star\|^2_{\V_\star} d\xi+ C \int_\R |\xi|^{2\beta} \|{\mathcal F}\tilde{\boldsymbol{\mathfrak{u}}}_\star\|^2_{\V^{-\alpha}_\star} d\xi,
\end{array}
$$
where Plancherel's equality and the continuous embedding between  $\V_\star$ and $\V_\star^{-\alpha}$ were used in the last line. The above estimate also holds trivially for $\beta < 0$. Thus we get
$$
\|\partial_t \tilde{\boldsymbol{\mathfrak{u}}}_\star \|_{H^{\beta-1}(\R; \V_\star^{-\alpha})}+\|{\tilde{\boldsymbol{\mathfrak{u}}}}_\star\|_{H^{\beta}(\R; \V_\star^{-\alpha})}\le C.
$$    
As a result of \eqref{W^s<V^s} for $s\in[0,\frac{1}{2})$,  we obtain    
$$
\|\partial_t \tilde{\boldsymbol{\mathfrak{u}}}_\star \|_{H^{\beta-1}(0,T; \W_\star^{-\alpha})}+\|\tilde{\boldsymbol{\mathfrak{u}}}_\star\|_{H^{\beta}(0,T; \W_\star^{-\alpha})}\le C,
$$
for all  $\alpha $ satisfying $0\le\alpha\le s\le 1+2\alpha<2$, and for all $\beta$ satisfying $\beta<\bar\beta:=\frac{1+\alpha}{1+s}(\frac{s}{2}+\frac{1}{4})$. This latter inequality leads to \eqref{thm5.6-lab1}. 

Next, multiply \eqref{thm5.6-lab3} by $A_\star^{1-s}\mathcal{F}\tilde{\boldsymbol{\mathfrak{u}}}_\star$ and take the real part to get
$$
\nu\|A_\star\mathcal{F}\tilde{\boldsymbol{\mathfrak{u}}}_\star\|^2_{\V_\star^{-s}}\le C\|{\mathcal F}\tilde{\boldsymbol{\mathfrak{g}}}\|_{\W^{-s}_\star}\|A_\star^{1-s}\mathcal{F}\tilde{\boldsymbol{\mathfrak{u}}}_\star\|_{\V_\star^{s}}\le C\|{\mathcal F}\tilde{\boldsymbol{\mathfrak{g}}}\|_{\W^{-s}_\star}\|A_\star\mathcal{F}\tilde{\boldsymbol{\mathfrak{u}}}_\star\|_{\V_\star^{s}},
$$
where we have applied \eqref{W^s<V^s} for $s\in[0,2)$.
Thus, 
$$
(1+|\xi|)^{-2r}\|A_\star\mathcal{F}\tilde{\boldsymbol{\mathfrak{u}}}_\star\|_{\V_\star^{-s}}^2\le C (1+|\xi|)^{-2r} \|{\mathcal F} \tilde{\boldsymbol{\mathfrak{g}}}\|_{\W^{-s}_\star}^2, 
$$
and hence
$$
\|A_\star\tilde{\boldsymbol{\mathfrak{u}}}_\star\|_{H^{-r}(\R; \V_\star^{-s})}\le C \|\tilde{\boldsymbol{\mathfrak{g}}}\|_{H^{-r}(\R; \W_\star^{-s})}.
$$
It is not hard to see from \eqref{W^s<V^s}  that $\|\Delta_\star\v_\star\|_{\W^{-s}_\star}\le C \|A_\star\v_\star\|$ for all $\v_\star\in \V_\star$ and all $s\in[0,\frac{3}{2})$. Then, using \eqref{thm5.6-lab5} yields   
$$
\|\Delta_\star\tilde{\boldsymbol{\mathfrak{u}}}_\star\|_{H^{-r}(0,T; \W^{-s}_\star)}\le C,
$$
for all $r>\bar r$, which implies \eqref{thm5.6-lab2}. 
\end{proof}
\begin{corollary}\label{co5.7} Suppose that assumptions $(\rm H1)$-$(\rm H4)$ hold. For $\alpha\in[\frac{1}{4}, \frac{1}{2})$ and $\beta<\bar\beta=\frac{2}{5}(1+\alpha)$,  it follows that
\begin{equation}\label{co5.7-lab1}
\|\partial_t\u_\star\|_{H^{\beta-1}(0,T; \W^{-\alpha}_\star)}+\|\u_\star\|_{H^{\beta}(0,T; \W^{-\alpha}_\star)}\le C,
\end{equation}
where $C>0$ is a constante independent of $h$.
\end{corollary}
\begin{proof} 
From $s\in[\frac{1}{2},\frac{3}{2})$ and $0\le\alpha\le s\le 1+2\alpha<2$, we obtain $\alpha\in[\frac{1}{4}, \frac{1}{2})$. Next note that $\frac{1}{1+s} (\frac{s}{2}+\frac{1}{4})$ reaches its maximum $\frac{2}{5}$ at $s=\frac{3}{2}$.  Therefore we can simplify the expression $\bar\beta$ in term of $\alpha$ only as $\bar\beta=\frac{2}{5}(1+\alpha)$ in \eqref{thm5.6-lab1}. 
\end{proof}

Using \eqref{co4.2-lab1}, one can also prove the following. 

\begin{corollary} Assume that assumptions $(\rm H1)$-$(\rm H4)$ hold. Then, for $\alpha\in[\frac{1}{4},\frac{1}{2})$ and $\beta<\bar\beta=\frac{2}{5}(1+\alpha)$, it follows that there exists a constant $C>0$, independent of $h$, such that 
\begin{equation}\label{co5.8-lab1}
\|\partial_t\u_h\|_{H^{\beta-1}(0,T; \tilde\H^{-\alpha}_0(\Omega))}+\|\u_h\|_{H^{\beta}(0,T; \tilde\H^{-\alpha}_0(\Omega))}\le C.
\end{equation}
Furthermore, for $s\in[\frac{1}{2},\frac{3}{2}]$ and $r$ such that $r>\bar r=\frac{3}{4}-\frac{s}{2}=\frac{1}{p}-\frac{1}{2}$, it follows that
\begin{equation}\label{co5.8-lab2}
\|\Delta_h\u_h\|_{H^{-r}(0,T; \tilde\H^{-s}_0(\Omega))}\le C.
\end{equation}
\end{corollary}
We now proceed to obtain an estimate for $p_h$.

\begin{lemma} Suppose that conditions $(\rm H1)$-$(\rm H4)$ hold. There exists a constant $C>0$, independent of $h$, such that, for $s\in[\frac{1}{2}, \frac{7}{10}]$ and $r>\bar r=\frac{3}{4}-\frac{s}{2}$,
\begin{equation}\label{lm5.9-lab1}
\|p_h\|_{H^{-r}(0,T; H^{1-s}(\Omega))}\le C,
\end{equation}
where $C>0$ is a constant independent of $h$.
\end{lemma}

\begin{proof} First we write \eqref{co5.7-lab1} as 
$$
\|\partial_t	\u_\star\|_{H^{-r}(0,T; \W^{-\alpha}_\star)}\le C, 
$$
where $\alpha\in[\frac{1}{4},\frac{1}{2})$ and $r>\tilde r:=1-\bar\beta=\frac{3}{5}-\frac{2}{5}\alpha$. As a result, we have that
\begin{equation}\label{lm5.9-lab2}
\|\partial_t	\u_\star\|_{H^{-r }(0,T; \W^{-s}_\star)}\le C 
\end{equation}
holds for $\alpha\le s$ and $\tilde r\le \bar r$ provided that $s\in[\frac{1}{2},\frac{7}{10}]$ and $r>\bar r=\frac{3}{4}-\frac{s}{2}$.    

From \eqref{Sec5-Lab1}, we bound
$$
\begin{array}{rcl}
\|p_h\|_{H^{1-s}(\Omega)}&\le& \displaystyle \sup _{\v_\star\in \W_\star\backslash\{ 0\}} \frac{(\nabla p_h, \v_\star)}{\|\v_\star\|_{\W^s_\star}}
\\
&\le&C(\|\partial_t \u_\star\|_{\W^{-s}_\star }+\|\Delta_\star\u_\star\|_{\W^{-s}_\star}
+\|{\mathcal N}_\star(\u_\star,\u_\star)\|_{\W^{-s}_\star}+\|\f_h\|_{\W^{-s}_\star}
)
\\
&\le&C(\|\partial_t \u_\star\|_{\W^{-\alpha}_\star }+\|\Delta_\star\u_\star\|_{\W^{-s}_\star}
+\|{\mathcal N}_\star(\u_\star,\u_\star)\|_{\W^{-s}_\star}+\|\f_h\|_{\tilde\H^{-s}_0(\Omega)}).
\end{array}
$$

The proof is completed via \eqref{lm5.9-lab2}, \eqref{thm5.6-lab2}, \eqref{lm5.5-lab1} and \eqref{rm3.1-lab2}.

\end{proof}

\section{Convergence towards weak and suitable weak solutions}\label{weaksol}
In this section we will prove that the sequence of the approximate solutions provided by scheme \eqref{eq:GalerkinNS} converges towards a weak solution in the sense of Definition \ref{def:weak-solution3} and towards a suitable weak solution in the sense of Definition \ref{Suitable_solution}. In order for these convergence results to hold, we will need to use the following compactness results {\it à la Aubin-Lions}.

The following compactness result is due to Lions \cite{Lions_1969}.
\begin{lemma}\label{Lm6.1} Let $H_0\hookrightarrow H \hookrightarrow H_1$ be three Hilbert spaces with dense and continuous embedding. Assume that the embedding $H_0\hookrightarrow H$ is compact. Then $L^2(0,T; H_0)\cap H^\gamma (0,T; H_1) $ embeds compactly in $L^2(0,T; H)$ for $\gamma>0$. 
\end{lemma}
The proof of the two following compactness result can be found in \cite[Ap. A.1, A.2]{Guermond_2007}.
\begin{lemma}\label{Lm6.2} Let $X\hookrightarrow Y$ be two Hilbert spaces with compact embedding. Then $H^\beta(0,T; X)$ embeds continuously and compactly in $C^0([0,T]; Y)$ for $\beta>\frac{1}{2}$.   
\end{lemma}
\begin{lemma}\label{Lm6.3} Let $H_0\hookrightarrow H_1$ be two Hilbert spaces with compact embedding. Let $\gamma>0$ and $\gamma>\mu$, then the injection $H^\gamma(0,T; H_0)\hookrightarrow H^\mu (0,T H_1)$ is compact. 
\end{lemma}

\begin{theorem}\label{Thm6.4} Assume that hypotheses $\rm (H1)$-$\rm(H4)$ are satisfied. Then there exists a subsequence (denoted in the same way) of approximate solutions $(\u_h, p_h)$ converging toward a weak solution given in  Definition \eqref{def:weak-solution3} in the following sense as $h\to 0$: 
\begin{equation}\label{Thm6.3-lab1}
\u_h \to \u \mbox{ in } L^2(0,T; \H^1_0(\Omega))-weak \mbox{ and in } L^2(0,T; \H^\beta(\Omega))-strong \mbox{ for all }\beta<1     
\end{equation}
and
\begin{equation}\label{Thm6.3-lab2}
p_h\to p \mbox{ in } H^{-r} (0,T; H^\delta(\Omega))-weak  \mbox{ for all } \delta\in[\frac{3}{10}, \frac{1}{2}] \mbox{ and } r>\frac{1}{4}+\frac{\delta}{2}.  
\end{equation}
\end{theorem}
\begin{proof} Let $\v\in H^r(0,T; \tilde\H^s_0(\Omega))$, for $s\in (\frac{1}{2}, \frac{7}{10}]$ and $ r>\frac{3}{4}-\frac{s}{2}$, and $q\in L^2(0,T; H^1_{\int=0}(\Omega))$. From \eqref{approx-velocity} and \eqref{approx-pressure}, we are allowed to construct three sequences $\{\v_h\}_{h>0}\subset H^r(0,T; \W_h)$, $\{\tilde\v_h\}_{h>0}\subset H^r(0,T; \tilde\W_h)$ and $\{q_h\}_{h>0}\subset L^2(0,T; Q_h)$ such that $\v_h\to \v$ in $H^r(0,T;\tilde\H^s_0(\Omega))$-strong, $\tilde\v_h\to \boldsymbol{0}$ in $L^2(0,T; \L^2(\Omega))$-strong and $q_h\to q$ in $L^2(0,T; H^1_{\int=0}(\Omega))$-strong as $h\to0$.

By virtue of \eqref{energy-estimate}, \eqref{co5.8-lab1}, \eqref{co5.8-lab2} and \eqref{lm5.9-lab1}, we know that there exist a subsequence of $\{\v_h\}_{h>0}$ and $\{p_h\}_{h>0}$, still denoted by itself, and a pair $(\u, p)$ such that
$$
\u_h \to \u \mbox{ in } L^\infty(0,T; \L^2(\Omega))-\mbox{weak-$\star$},     
$$
$$
\u_h \to \u \mbox{ in } L^2(0,T; \H^1_0(\Omega))-\mbox{weak},     
$$
$$
\partial_t\u_h \to \partial_t\u \mbox{ in } H^{-r}(0,T; \tilde\H^{-s}_0(\Omega))-\mbox{weak}, 
$$
$$
\Delta_h\u_h \to \Delta\u \mbox{ in } H^{-r}(0,T; \tilde\H^{-s}_0(\Omega))-\mbox{weak}, 
$$
and
$$
\nabla p_h\to \nabla p \mbox{ in } H^{-r} (0,T; \tilde\H^{-s}_0(\Omega))-\mbox{weak},  
$$
for all $s\in(\frac{1}{2}, \frac{7}{10}]$ and  $ r>\frac{3}{4}-\frac{s}{2}$. Observe that we have used that the fact that $\tilde\H^{-s}(\Omega)$ coincides with $\H^{-s}_0(\Omega)$ for $s\in(\frac{1}{4}, \frac{7}{10}]$ for the pressure.  We also have, from \eqref{energy-estimate}, that
\begin{equation}\label{Thm6.4-lab1}
\tilde\u_h\to 0\quad \mbox{ in }\quad L^2(0,T; \L^2(\Omega))-\mbox{strong},
\end{equation}
since
$$
\frac{\nu^\frac{1}{2}}{h}\|\tilde\u_h\|_{L^2(0,T;\L^2(\Omega))}\le \|\tau^{-\frac{1}{2}}\tilde\u_h\|_{L^2(0,T; \L^2(\Omega))}\le C.
$$

We can pass to the limit in \eqref{eq:Galerkin-p}. Thus we find that $\nabla\cdot \u=0$ in $(0,T)\times \Omega$, whence $\u\in L^\infty (0,T; \H)\cap L^2(0,T; \V)$. 
For the trilinear terms, we proceed as follows.  By Lemma \ref{Lm6.1},  we have that 
$$
\u_h \to \u \mbox{ in } L^2(0,T; \H^\beta(\Omega))-strong \mbox{ for all }\beta<1,     
$$
since $\{\u_h\}_{h>0}$ is bounded in  $L^2(0,T; \H^1_0(\Omega))\cap H^{\beta}((0,T);\tilde\H^{-\alpha}_0(\Omega))$ for $\alpha\in[\frac{1}{4},\frac{1}{2})$ and $0<\beta <\frac{2}{5}(1+\alpha)$ from \eqref{energy-estimate} and  \eqref{co5.8-lab1}. 
Therefore,
$$\mathcal{N}(\u_h,\u_h)\to \mathcal{N}(\u,\u)\mbox{ in } \mathcal{D}' ((0,T)\times\Omega).$$
As a consequence of \eqref{lm5.5-lab1}, we obtain            
$$\mathcal{N}(\u_h,\u_h)\to \mathcal{N}(\u,\u)\mbox{ in } H^{-r}(0,T;  \tilde\H^{-s}_0(\Omega)).$$
On passing to the limit in \eqref{eq:Galerkin-p}, we have had that $\nabla\cdot\u=0$ in $(0,T)\times \Omega$, thereby 
$$\mathcal{N}(\u_h,\u_h)\to (\u\cdot\nabla)\u\mbox{ in } H^{-r}(0,T;  \tilde\H^{-s}_0(\Omega)).$$
By an analogous argument, we find that 
$$\tilde{\mathcal{N}}(\u_h, \tilde\u_h)\to \boldsymbol{0}\mbox{ in } H^{-r}(0,T;  \tilde\H^{-s}_0(\Omega)),$$
where $\langle \tilde{\mathcal{N}}(\u_h,\tilde\w_h), \v_h \rangle=b(\u_h, \v_h, \tilde\w_h)$ for all $\u_h,\w_h\in \W_h$ and $\tilde\v_h\in\tilde\W_h$.

From the above convergences, it is easy to see that 
$$
\partial_t \u_h+\mathcal{N}(\u_h,\u_h)-\nu\Delta_h\u_h +\nabla p_h-\tilde{\mathcal{N}}(\u_h, \tilde\u_h)-\f_h \to  \partial_t \u+(\u_h\cdot\nabla)\u_h-\nu \Delta\u+\nabla p-\f.
$$
in $H^{-r}(0,T; \tilde \H^{-s}_0(\Omega))$ as $h\to 0$.

For the initial condition,  we have that $\u_h\to \u$ in $C^0([0,T]; \tilde\H^{-\alpha}_0(\Omega))$-strong for $\alpha\in (\frac{1}{4},\frac{1}{2})$ by Lemma~\ref{Lm6.2}; therefore, $\u_h(0)\to \u(0)$ in $\tilde\H^{-\alpha}_0(\Omega)$. Furthermore, it follows from \eqref{proj_u0} and \eqref{lm4.2-lab3} that $\u_{0h}\to\u_0$ in $\tilde\H^{-\alpha}_0(\Omega)$. We have thus shown that $\u(0)=\u_0$.

The energy inequality can be verified by the lower semicontinuity of the norm for the weak topology; for complete details, see  \cite{Badia-Gutierrez}.             
\end{proof}

\begin{theorem} Under hypotheses $\rm (H1)$-$\rm(H4)$, the sequence of approximate solutions $(\u_h, p_h)$ converges, up to a subsequence, to a suitable weak solution given in Definition \ref{Suitable_solution} as $h\to0$.  
\end{theorem}
\begin{proof} Let $\phi\in \mathcal{D}((0,T)\times\Omega; \R^+)$ and substitute $\v_h=\pi_{\W_h}(\u_h\phi)$ into \eqref{eq:Galerkin-u} 
to get
\begin{equation}\label{Thm6.5-lab1}
\begin{array}{l}
\displaystyle \int_0^T \{(\partial_t\u_h,\pi_{\W_h}(\u_h\phi))+b(\u_h,\u_h,
\pi_{\W_h}(\u_h\phi))+\nu(\nabla\u_h,\nabla\pi_{\W_h}(\u_h\phi))
\\
\displaystyle
-(p_h,\nabla\cdot\pi_{\W_h}(\u_h\phi))- b(\u_h,\pi_{\W_h}(\u_h\phi),\tilde\u_h)-(\f_h,\pi_{\W_h}(\u_h\phi))\}\,\dt=0.
\end{array}
\end{equation}

We are ready to take the limit in \eqref{Thm6.5-lab1} as $h\to0$ so as to prove that the weak solution $(\u, p)$ found in Theorem \ref{Thm6.4} is suitable. We will only focus on passing to the limit in the terms of \eqref{Thm6.5-lab1} involving the subscale velocity $\tilde\u_h$ and the pressure term. The remaining terms appear in a rudimentary  finite element formulation so that a proof can be found in \cite{Guermond_2009}. In particular, from \eqref{comm-prop-vel} and \eqref{inverseH^s-L^2andL^2-H^-s} and in virtue of Lemma \ref{Lm6.3}, it follows that 
$$
\lim_{h\to0}\int_0^T (\partial_t\u_h,\u_h\phi)\,\dt= -\frac{1}{2}\int_0^T (|\u|^2,\partial_t\phi)\,\dt,
$$
$$
\lim_{h\to0}\int_0^T b(\u_h,\u_h,\pi_{\W_h}(\u_h\phi))\, \dt= -\frac{1}{2}\int_0^T (|\u|^2\u, \nabla\phi)\,dt, 
$$   
$$
 \liminf_{h\to0} \int_0^T\nu(\nabla\u_h,\nabla\pi_{\W_h}(\u_h\phi))\,\dt\ge \int_0^T (|\nabla\u|^2, \phi)\, \dt-\int_{0}^T(\frac{1}{2}|\u|^2, \Delta\phi)\,\dt,
$$
and
$$
\lim_{h\to0} -\int_0^T\langle\f_h,\u_h\phi\rangle\,\dt=-\int_0^T \langle\f,\u\phi\rangle\, \dt.
$$

To begin with, we first turn our attention to passing to the limit in the convective term.
\begin{equation}\label{Thm6.5-lab5}
\begin{array}{rcl}
b(\u_h,\pi_{\W_h}(\u_h\phi),\tilde\u_h)\, \dt&=&b(\u_h,\u_h\phi,\tilde\u_h)\,\dt+b(\u_h,\pi_{\W_h}(\u_h\phi)-\u_h\phi,\tilde\u_h)\,\dt
\\
&=&  b (\u_h, \u_h, \tilde\u_h\phi)+(\u_h\cdot\nabla\phi \u_h, \tilde\u_h)+b(\u_h,\pi_{\W_h}^\perp(\u_h\phi),\tilde\u_h)
\\ 
&=& (\pi^\perp_{\W_h} (\mathcal{N}(\u_h,\u_h)\phi), \tilde\u_h)+ (\pi_{\W_h}^\perp(\u_h\cdot\nabla\phi \u_h), \tilde\u_h)
\\
&&+b(\u_h,\pi_{\W_h}^\perp(\u_h\phi),\tilde\u_h).
\end{array}
\end{equation}
From \eqref{comm-prop-vel} and \eqref{inverseL^inf-L^k}, we have:
\begin{align*}
\int_0^T (\pi^\perp_{\W_h} (\mathcal{N}(\u_h,\u_h)\phi), \tilde\u_h)\, \dt\le& \int_0^T \|\pi^\perp_{\W_h} (\mathcal{N}(\u_h,\u_h)\phi), \tilde\u_h)\| \|\tilde\u_h\| \dt\le C  \int_0^T h \|\u_h\|_{\L^\infty(\Omega)} \|\nabla\u_h\| \|\tilde\u_h\|\, \dt
\\
\le& C  \left(\int_0^T \tau h^2 \|\u_h\|^2_{\L^\infty(\Omega)} \|\nabla\u_h\|^2\,\dt \right)^{\frac{1}{2}} \left(\int_0^T \tau^{-1} \|\tilde\u_h\|^2\,\dt\right)^{\frac{1}{2}}
\\
\le& C h^{\frac{3}{4}} \|\u_h\|^\frac{1}{2}_{L^\infty(0,T; \L^2(\Omega))} \|\u_h\|_{L^2(0,T; \H^1_0(\Omega))} \|\tilde\u_h\|_{\tau^{-\frac{1}{2}} L^2(0,T; \L^2(\Omega))}
\end{align*}  
and hence
$$
\lim_{h\to 0} \int_0^T (\pi^\perp_{\W_h} (\mathcal{N}(\u_h,\u_h)\phi), \tilde\u_h)\, \dt=0.
$$
Analogously, we bound
$$
\int_0^T (\u_h\cdot\nabla\phi, \u_h\cdot\tilde\u_h)\,\dt\le C T h^{\frac{3}{4}} \|\u_h\|_{L^\infty(0,T; \L^2(\Omega))} \|\tilde\u_h\|_{\tau^{-\frac{1}{2}}L(0,T; \L^2(\Omega))} 
$$
and
$$
\int_0^T b(\u_h,\pi_{\W_h}(\u_h\phi)-\u_h\phi,\tilde\u_h)\,\dt\le C h^{\frac{3}{4}} \|\u_h\|^\frac{1}{2}_{L^\infty(0,T; \L^2(\Omega))} \|\u_h\|_{L^2(0,T; \H^1_0(\Omega))} \|\tilde\u_h\|_{\tau^{-\frac{1}{2}} L^2(0,T; \L^2(\Omega))}.
$$
Thus
$$
\lim_{h\to0} \int_0^T (\u_h\cdot\nabla\phi, \u_h\cdot\tilde\u_h)\,\dt=0,
$$
and  
$$
\lim_{h\to 0}\int_0^T b(\u_h,\pi_{\W_h}(\u_h\phi)-\u_h\phi,\tilde\u_h)\,dt=0. 
$$

For the ``viscous" term, it is not hard to see that
$$
\liminf_{h\to 0}\int_0^T \tau^{-1}(|\tilde\u_h|^2,\phi)\,\dt\ge0.
$$ 

For the pressure terms, we write 
$$
\int_0^T (p_h,\nabla\cdot \pi_{\W_h}(\u_h\phi))\,\dt=\int_0^T (p_h\u_h, \nabla\phi)\,dt+\int_0^T (p_h, \nabla \cdot (\pi_{\W_h}(\u_h\phi)-(\u_h\phi)))\,\dt+\int_0^T (\phi p_h,\nabla\cdot\u_h)\,\dt.
$$
It was proved in \cite{Guermond_2007} that 
$$
\lim_{h\to 0}  \int_0^T (p_h \u_h,\nabla\phi)\,\dt=\int_0^T (p u, \nabla\phi)\,\dt
$$
and
$$
\lim_{h\to0}=\int_0^T (p_h, \nabla \cdot (\pi_{\W_h}(\u_h\phi)-(\u_h\phi)))\,\dt=0. 
$$
For the remaining pressure terms, we use \eqref{eq:Galerkin-p} with $q_h=\pi_{Q_h} (\phi p_h)$ to obtain
\begin{align*}
\int_0^T (\phi p_h,\nabla\cdot\u_h)\,\dt&+\int_0^T (\nabla p_h, \tilde\u_h\phi)=\int_0^T (p_h\phi-\pi_{Q_h} (p_h\phi), \nabla\cdot\u_h)\,\dt
\\
&+\int_0^T (\nabla(p_h\phi)-\nabla\pi_{Q_h}(p_h\phi), \tilde\u_h)\,\dt-\int_0^T (p_h\nabla\phi, \tilde\u_h)\,\dt. 
\end{align*}
We know from \cite{Guermond_2007} that 
$$
\lim_{h\to0} \int_0^T (p_h\phi-\pi_{Q_h} (p_h\phi), \nabla\cdot\u_h)\,\dt=0.
$$

Let $\varepsilon>0$  and set $s=\frac{1}{2}+\frac{16}{9}\varepsilon$ and $\bar r=\frac{3}{4}-\frac{s}{2}=\frac{1}{4}-\frac{4}{9}\varepsilon$. Now choose $r=\frac{1}{2}-\frac{4}{9}\varepsilon$. Moreover, set $\alpha=\frac{1}{4}-\frac{5}{9}\varepsilon$ and $\bar\beta=\frac{2}{5} (1+\alpha)=\frac{1}{5}-\frac{2}{9}\varepsilon$. Thus we have $1-s>\alpha$ and $\bar\beta>r$ since
$$
1-s=\frac{1}{2}-\frac{16}{9}\varepsilon>\frac{1}{2}(\frac{1}{2}-\frac{16}{9}\varepsilon)=\frac{1}{4}-\frac{4}{9}\varepsilon>\frac{1}{4}-\frac{5}{9}\varepsilon=\alpha
$$
and
$$
r=\frac{1}{2}-\frac{4}{9}\varepsilon<\frac{1}{2}-\frac{2}{9}\varepsilon=\frac{2}{5}(\frac{5}{4}-\frac{5}{9}\varepsilon)=\frac{2}{5}(1+\frac{1}{4}-\frac{5}{9}\varepsilon)=\frac{2}{5}(1+\alpha)=\bar\beta.
$$
From the a priori energy estimates \eqref{co5.7-lab1} and \eqref{lm5.9-lab1} and the commutator property \eqref{comm-prop-pre}, our choice of parameters yields
\begin{align*}
\int_0^T (\nabla(p \phi)-\nabla\pi_{Q_h}(p_h\phi),\tilde\u_h)\,\dt&\le \|\nabla(p \phi)-\nabla\pi_{Q_h}(p_h\phi)\|_{H^{-r}(0,T; \L^2(\Omega))} \|\tilde\u_h\|_{H^r(0,T; \L^2(\Omega))} 
\\
&\le C {h^{1-s-\alpha}} \|p_h\|_{H^{-r}(0,T; H^{1-s}(\Omega))} \|\tilde\u_h\|_{h^{\alpha}H^r(0,T;\L^2(\Omega))}
\end{align*}
and hence
$$
\lim_{h\to0} \int_0^T (\nabla(\phi)-\nabla\pi_{Q_h}(p_h\phi),\tilde\u_h)\,\dt=0.
$$
Finally, it is easy to see in a similar fashion that 
$$
\lim_{h\to 0}\int_0^T (\phi p_h,\nabla\cdot\u_h)\,\dt=\lim_{h\to} \int_0^T (\phi p_h-\pi_{\W_h}(\phi p_h),\nabla\cdot\u_h)\,\dt=0.
$$

\end{proof}
\appendix
\section{Proof of the inverse inequalities \eqref{inverseH^1-H^s}}\label{App-A}

To prove inequalities \eqref{inverseH^1-H^s}, we follow very closely the arguments developed in \cite[Thm. 4.5.11]{Brenner-Scott}. 

We first need to introduce an equivalent norm for fractional order Hilbert spaces as follows. Let $s\in(0,1)$. Then 
$$
\|u\|^2_{H^s(\Omega)}=\|u\|^2+|u|^2_{H^s(\Omega)},
$$
where
$$
|u|^2_{H^s(\Omega)}=\int_\Omega\int_\Omega \frac{|u(\x)-u(\y)|^2}{|\x-\y|^{3+2s}}\,{\rm d}\x\,{\rm d}\y.
$$

Given $(K,\mathcal{P}, \Sigma)$, we define $(\tilde K, \tilde P, \tilde\Sigma)$ where $\hat K=\{(1/h_K)\x : \x\in K\}$. Thus, if $u_h$ is a function defined on $K$, then $\hat u_h$ is defined on $ \tilde K$ by
$$
\hat u(\hat \x)=u ( h_K^{-1}\x)\quad \mbox{ for all } \quad \hat \x\in \hat K. 
$$
Thus we can write
$$
\|\nabla u_h\|_{\L^2(K)}= h_K^{\frac{1}{2}} \|\hat \nabla\hat u_h\|_{\L^2(K)}.
$$ 
As $\hat \nabla\hat u_h$ belongs to  a space of finite and fixed dimension on $\hat K$, on which all norms are equivalent, it is not hard to see that there is a constant $ C_{\hat T}>0$ such that 
$$
\|\hat \nabla\hat u_h\|_{\L^2(K)}\le  C_{\hat T} |\hat u_h|_{H^s(\hat K)}.
$$
Reverting to $K$, this leads to
$$
\|\hat \nabla \hat u\|_{\L^2}\le   C_{\hat T} h_K^{-\frac{3}{2}+s} |u_h|_{H^{s}(K)} 
$$
and hence
$$
\|\nabla u_h\|_{\L^2(K)}\le C_{\hat T} h^{-1+s}_K \|u_h\|_{H^s(\Omega)}.
$$
An argument in the proof of \cite[Prop. 4.4.11 ]{Brenner-Scott} shows that if $(\tilde K, \tilde{\mathcal{P}}, \tilde\Sigma)$ is a referent element, we have that there exists a constant $C_{\tilde T}>0$ such that $ C_{\hat T}\le C_{\tilde T}$.  
Summing over all elements $K$ and using the quasi-uniformity of the mesh leads to 
$$
\|\nabla u_h\|\le C h^{-1+s} \left(\sum_{K\in\mathcal{T}_h} \|u_h\|_{H^s(\Omega)}^2\right).
$$
Then \eqref{inverseH^1-H^s} follows because the sum of the fractional norms over all elements is smaller than the fractional norm over the union of the elements.
\section*{Acknowledgment}
The authors are very grateful to Professor Vivette Girault who provided a proof of a particular case of inequality \eqref{inverseH^1-H^s}. 
\bibliography{art17.bbl}
\bibliographystyle{plain}
\nocite{*}
\end{document}